\documentclass[a4paper]{article}
\usepackage[all]{xy}\usepackage[latin1]{inputenc}        
\usepackage[dvips]{graphics,graphicx}
\usepackage{amsfonts,amssymb,amsmath,color,mathrsfs, amstext}
\usepackage{amsbsy, amsopn, amscd, amsxtra, amsthm,authblk}
\usepackage{enumerate,algorithmic,algorithm}
\usepackage{upref}
\usepackage{geometry}
\usepackage{listings}
\usepackage[displaymath]{lineno}
\usepackage{float}
\usepackage{yhmath}
\usepackage{booktabs}
\usepackage{subcaption}
\usepackage{multirow}
\usepackage{makecell}
\usepackage[colorlinks,
            linkcolor=red,
            anchorcolor=red,
            citecolor=red
            ]{hyperref}
\usepackage{tikz}
\usepackage{float}
\usetikzlibrary{shapes.geometric, arrows.meta, positioning}
\numberwithin{equation}{section}

\newtheorem{theorem}{Theorem}[section]
\newtheorem{lemma}{Lemma}[section]

\newtheorem{corollary}{Corollary}[section]

\newtheorem{assumption}{Assumption}[section]
\makeatletter
\@mparswitchfalse
\makeatother

\title{Mean field error estimate of the random batch method for large interacting particle system}

\author[1]{Zhenyu Huang\thanks{E-mail: zhenyuhuang@sjtu.edu.cn}}

\author[2]{Shi Jin \thanks{E-mail: shijin-m@sjtu.edu.cn}}

\author[2,3]{Lei Li \thanks{E-mail: leili2010@sjtu.edu.cn}}
  
\affil[1]{School of Mathematical Sciences, Shanghai Jiao Tong University, Shanghai, 200240, P.R.China}

\affil[2]{School of Mathematical Sciences, Institute of Natural Sciences, MOE-LSC, Shanghai Jiao Tong University, Shanghai, 200240, P.R.China}

\affil[3]{Shanghai Artificial Intelligence Laboratory, Shanghai, P.R.China}

\date{}

\begin{document}

\maketitle

\begin{abstract}
The random batch method (RBM) proposed in  [Jin et al., J. Comput. Phys., 400(2020), 108877] for large interacting particle systems is an efficient with linear complexity in particle numbers and highly scalable algorithm for $N$-particle interacting systems and their mean-field limits when $N$ is large.  We consider in this work the quantitative error estimate of  RBM toward its mean-field limit, the Fokker-Planck equation. Under mild assumptions, we obtain a uniform-in-time $O(\tau^2+1/N)$ bound on the scaled relative entropy between the joint law of the random batch particles and the tensorized law at the mean-field limit, where $\tau$ is the time step size and $N$ is the number of particles. Therefore, we improve the existing rate in discretization step size from $O(\sqrt{\tau})$ to $O(\tau)$ in terms of the Wasserstein distance.
\end{abstract}

{\it Keywords: relative entropy, random batch method, interacting particle system, propagation of chaos}

\section{Introduction}\label{sec1}
Interacting particle systems arise in a variety of important problems in physical, social, and biological sciences, for example, molecular dynamics \cite{frenkel2023understanding}, swarming \cite{carlen2013kinetic,carrillo2019particle,degond2017coagulation,vicsek1995novel}, chemotaxis \cite{bertozzi2012characterization,horstmann20031970}, flocking \cite{albi2013binary,cucker2007emergent,ha2009simple}, synchronization \cite{choi2011complete,ha2014complete}, consensus \cite{motsch2014heterophilious} and random vortex model \cite{majda2002vorticity}. In this paper, we consider the following general first order system of $N$ particles:
\begin{equation}\label{firstorder}
d X^i=b\left(X^i\right) d t+\frac{1}{N-1} \sum_{j: j \neq i} K\left(X^i-X^j\right) d t+\sqrt{2\sigma} d W^i, \quad i=1,2, \ldots, N
\end{equation}
Here $X^i \in \mathbb{R}^d$ are the labels for particles, $b: \mathbb{R}^d \to \mathbb{R}^d$ is the drift force, $K: \mathbb{R}^d \to \mathbb{R}^d$ is the interaction kernel, $\sigma > 0$ is the diffusion coefficient and $W^i$ are $N$ independent $d$-dimensional Wiener processes (standard Brownian motions). We assume that the initial values $X^i (0) =: X^i_0$ are drawn independently from the same distribution $\rho_0$.

As is well known, under certain conditions, the mean field limit (i.e., $N \rightarrow \infty$ ) of (\ref{firstorder}) is given by the following nonlinear Fokker-Planck equation:
\begin{equation}\label{FPE}
    \partial_t \rho=-\nabla \cdot((b(x)+K * \rho) \rho)+\sigma \Delta \rho,
\end{equation}
where $\rho(t,x)$ is the particle density distribution. The convergence of the interacting particle system (\ref{firstorder}) towards the Fokker-Planck equation (\ref{FPE}) as $N \to \infty$ has been systemically studied in \cite{kac1956foundations,mckean1967propagation,jabinwang2017mean,golse2016dynamics,chaintron2022propagation,chaintron2022propagation2}. 

If one numerically discretizes (\ref{firstorder}) directly, the computational cost per time step is $O(N^2)$, which is prohibititively expensive for large $N$. The Random Batch Method (RBM) proposed in \cite{jin2020random} is a simple and generic random algorithm to reduce the computation cost per time step from $O(N^2)$ to $O(N)$. The idea was to utilize small but random batch, so that interactions only occur inside the small batches at each time step. This random mini batch idea was the key component of  the so-called stochastic gradient descent (SGD) \cite{robbins1951stochastic, bubeck2015convex} in machine learning. Due to the simplicity and scalability, it already has a variety of applications in solving the Poisson-Nernst-Planck, Poisson-Boltzmann and Fokker-Planck-Landau equations \cite{li2020some, carrillo2021random}, efficient sampling \cite{doi:10.1137/19M1302077}, molecular simulation \cite{doi:10.1137/20M1371385,liang2022superscalability}, and quantum Monte-Carlo method \cite{jinlixian2020random}. Readers can refer to the review article \cite{jinli2021random}.

The RBM algorithm corresponding to (\ref{firstorder}) is given in Algorithm \ref{algo1}. Suppose one aims to do the simulation until time $T>0$. One first chooses a time step $\tau>0$ and a batch size $p \ll N, p \geqslant 2$ that divides $N$. Define the discrete time  $t_k:=k \tau, k \in \mathbb{N}$. For each time subinterval $\left[t_{k-1}, t_k\right)$, there are two steps: (1) at time $t_{k-1}$, one divides the $N$ particles into $n:=N / p$ groups (batches) randomly; (2) the particles evolve with interaction inside the batches only. 

\begin{algorithm}
    \caption{The Random Batch Method (RBM)}\label{algo1}
   \begin{algorithmic}[1]
    \FOR{$k$ = $1:[T / \tau]$} 
    \STATE Divide $\{1,2, \ldots, N\}$ into $n=N / p$ batches randomly. 
    \FOR{each batch $\xi_k$} 
    \STATE Update $\bar{X}^i$ 's $\left(i \in \xi_k\right)$ by solving the following stochastic differential equation (SDE) with $t \in\left[t_{k-1}, t_k\right):$ 
        \begin{equation}\label{rbmfirst}
         d \bar{X}^i=b\left(\bar{X}^i\right) d t+\frac{1}{p-1} \sum_{j \in \xi_k, j \neq i} K\left(\bar{X}^i-\bar{X}^j\right) d t+\sqrt{2\sigma} d W^i.
         \end{equation}
    \ENDFOR
    \ENDFOR 
\end{algorithmic}
\end{algorithm}

The RBM is not only an efficient numerical method for approximating the original particle system \eqref{firstorder}, but also, when $N$ is large, a numerical particle method for the Fokker-Planck equation \eqref{FPE}.  The goal of this paper is to provide an explicit bound  $O(\tau^2+1/N)$ of the rescaled relative entropy between the random batch system (\ref{rbmfirst}) and the Fokker-Planck equation (\ref{FPE}), as given in Theorem \ref{thm}. More precisely, denote
$$
\bar{\rho}_t^N = \text{Law}(\bar{X}_t^1, \cdots, \bar{X}_t^N),
$$
here $\bar{X}_t^1, \cdots, \bar{X}_t^N$ are the continuous version of the random batch particles defined below, which coincide with the Euler-Maruyama discretization of (\ref{rbmfirst}) at time $T_k$:
\begin{equation}\label{ctsem}
    \bar{X}_{t}^i=\bar{X}_{T_k}^i+(t-T_k) b\left(\bar{X}_{T_k}^i\right)+\frac{(t-T_k)}{p-1}\sum_{j\in\xi_k} K(\bar{X}_{T_k}^i-\bar{X}_{T_k}^j)+\sqrt{2 \sigma} (W_t^i-W_{T_k}^i),\quad t \in\left[T_k, T_{k+1}\right).
\end{equation} 
Set
$$
\rho^N_t\left(x_1, \ldots, x_N\right)=\rho_t^{\otimes N},
$$
where $\rho_t$ is the solution to  the mean field Fokker-Planck equation (\ref{FPE}). In this paper, we 
will prove the following uniform-in-time relative entropy bound
\begin{equation}
        \sup_{t} \mathcal{H}_N\left( \bar{\rho}^N_t \mid \rho^N_t \right) \leq \mathcal{H}_N\left( \bar{\rho}^N_0 \mid \rho^N_0 \right)+c_1 \tau^2 + \frac{c_2}{N},
\end{equation}
where the constants $c_1$ and $c_2$ are independent of $N$. Here, 
$$
\mathcal{H}_N\left( \bar{\rho}^N_t \mid \rho^N_t \right)=\frac{1}{N} \int_{\mathbb{R}^{d N}} \bar{\rho}^N_t\left( \textbf{x}\right) \log \frac{\bar{\rho}^N_t\left(\textbf{x}\right)}{\rho^N_t\left(\textbf{x}\right)} d \textbf{x},
$$
is the rescaled relative entropy. 

 Our proof consists of the following ingredients.  Firstly, we study the time evolution of the relative entropy of the joint law and the tensorized law based on the Markov property of the discrete random batch. In order to do so, like \cite{li2022sharp}, we derive a Liouville equation for the  joint law for the fixed batch version of the RBM, where the drift term is the backward conditional expectation at the previous step, conditioned on the current value of $N$ particles. Then one can define the joint law of some given $\xi_k$, namely $\bar{\varrho}_t^{N,\xi_k}$ by the Markov property. Since $$\bar{\rho}_{t}^N=\mathbb{E}_{\xi_k}\left[\bar{\varrho}_{t}^{N,\xi_k}\right],$$ therefore one can obtain the desired time evolution equation. Secondly, during the proof,  like in \cite{li2022sharp},  techniques including integration by parts (to eliminate the Gaussian noise that would possibly lower the convergence rate) and the Girsanov transform are used. We also obtain the estimate of the Fisher information for the joint law of the fixed batch version of the RBM, as was done in \cite{mou2022improved}. Thirdly, we use the Law of Large Number at the exponential scale as in \cite{jabin2018quantitative}.
 
 Several systems will be involved for our analysis.
 For clarity, in Fig.\ref{fig:enter-label}, we depict the relationships between these systems.

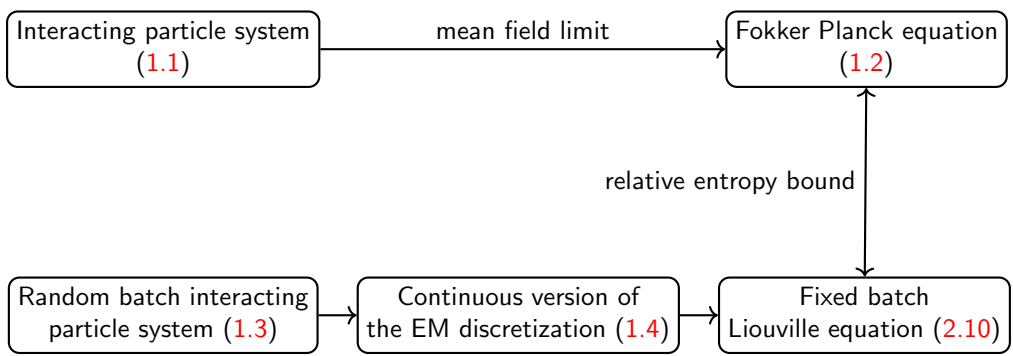
\begin{figure}[H]
    \centering
\begin{tikzpicture}[font=\sffamily, thick, node distance=0.5cm, 
    box/.style = {draw, rounded corners, align=center, minimum height=1cm}]
     \node[box, ] (RBMIPS) {Random batch interacting \\ particle system (\ref{rbmfirst})};
     \node[box, right=of RBMIPS ] (Euler) {Continuous version
     of\\ the EM discretization (\ref{ctsem})};
     \node[box, right=of Euler ] (Liou) {Fixed batch \\Liouville equation (\ref{liou})};
     \node[box, above=2.5cm of RBMIPS] (IPS) {Interacting particle system\\(\ref{firstorder})};
     \node[box, right=5.325 cm of IPS] (FP) {Fokker Planck equation\\(\ref{FPE})};
    \draw[->] (IPS) --node[above]  {mean field limit} (FP);
    \draw[->] (RBMIPS) -- (Euler);
    \draw[->] (Euler) -- (Liou);
    \draw[<->] (Liou) -- node[left]  {relative entropy bound} (FP);
\end{tikzpicture}
    \caption{Illustration of the various equations.}
    \label{fig:enter-label}
\end{figure}
Recently, there have been some theoretical results on the random batch method. The strong and weak error analysis for the RBM has been conducted in \cite{jin2021convergence} and they show that the strong error is of $O( \sqrt{\tau} )$ while the weak error is of $O(\tau)$ for interacting particles with  disparate species and weights. Moreover, the theoretical justification for the sampling accuracy is done in \cite{jin2023ergodicity} and they give the geometric ergodicity and the long time behavior of the RBM for interacting particle systems. They show that the Wasserstein distance between the invariant distributions of the interacting particle systems and the random batch interacting particle systems is bounded by $O(\sqrt{\tau})$. Besides, the error estimate and the long-time behavior of the discrete time approximation of random batch method was studied in \cite{ye2022error} using the triangle inequality framework, and  still  an $O(\sqrt{\tau})$ bound is obtained. We note that  the order of accuracy in the time step $\tau$ may not be optimal in the work \cite{jin2023ergodicity,ye2022error}. This is because they  used the strong error estimate of RBM.  The novelty of our result is to go beyond the existing strong error of RBM. We analyze the law of RBM at the level of the Liouville equation, enabling us to improve the order of convergence. It is worth mentioning that \cite{jin2022mean} also investigates the mean field limit of RBM, which gives rise to a mean field limit for $p$-body particle systems,  unlike the classical one-body mean field limit for interacting particle systems. Our work can be seen as an extension and improvement of the work \cite{jin2022mean}. More precisely, unlike the propagation of chaos in the classical mean field limit for interacting particle systems, the chaos in \cite{jin2022mean} is imposed at every discrete time. In other words, their argument of the mean field limit relies on the fact that two particles are unlikely to be related in RBM when $N \to \infty$ for finite iterations. Our results provide a uniform-in-time explicit error estimate between system (\ref{rbmfirst}) and the mean-field limit (\ref{FPE}) and directly prove the classical propagation of chaos for RBM.

The rest of the paper is organized as follows. We first introduce some preliminaries in Section \ref{sec2} including the Euler-Maruyama discretization of random batch and the Liouville equation for the joint law for the case of the fixed batch size. In Section \ref{sec3}, we present some auxiliary results which are useful in the proof of the main results. In Section \ref{mainthm}, we present the main result: the uniform-in-time estimate of the relative entropy between the joint law of the random batch particles and the tensorized law for the limit equation (the Fokker-Planck equation) and provide the detailed proof. Some technical details that are not so essential are moved to Appendix \ref{append}.


\section{Preliminaries}\label{sec2}
\subsection{Discrete random batch}
For the convenience of the analysis, we define
\begin{equation}
    K^{\xi_k}(\bar{X}_t^i) = \frac{1}{p-1}\sum_{j\in \xi_k,j \neq i}K(\bar{X}_t^i-\bar{X}_t^j), \quad t\in \left[T_k, T_{k+1}\right).
\end{equation}
Consider the Euler-Maruyama scheme with constant time step for (\ref{firstorder}):
\begin{equation}\label{EM}
    X^i_{T_{k+1}} = X^i_{T_{k}} + \tau b(X^i_{T_{k}}) + \frac{\tau}{N-1} \sum_{j: j\neq i} K(X^i_{T_{k}}-X^j_{T_{k}}) + \sqrt{2\sigma} \left(W_{T_{k+1}}^i-W_{T_k}^i\right).
\end{equation}
Similarly, the Euler-Maruyama scheme for the RBM  with time step $\tau$ at the $k$-th iteration for the $i$-th particle is:
\begin{equation}
    \bar{X}_{T_{k+1}}^i=\bar{X}_{T_k}^i+\tau b\left(\bar{X}_{T_k}^i\right)+\tau K^{\xi_k}(\bar{X}_{T_k}^i)+\sqrt{2 \sigma } \left(W_{T_{k+1}}^i-W_{T_k}^i\right),
\end{equation}
with $T_k:=k \tau$. Here, $\xi_k$ are i.i.d. sampled. Also we consider the following continuous version which coincides with the discrete RBM at grid point $T_k$ :
\begin{equation}\label{onerbm}
    \bar{X}_{t}^i=\bar{X}_{T_k}^i+(t-T_k) b\left(\bar{X}_{T_k}^i\right)+(t-T_k) K^{\xi_k}(\bar{X}_{T_k}^i)+\sqrt{2 \sigma} (W_t^i-W_{T_k}^i),\quad t \in\left[T_k, T_{k+1}\right).
\end{equation}
Denote $$\bar{\textbf{X}}_{T_{k}}=(\bar{X}^1_{T_k}, \cdots , \bar{X}^N_{T_k})^T \in \mathbb{R}^{Nd},$$ then the Euler-Maruyama scheme for the RBM  for $N$ particles with time step $\tau$ at the $k$-th iteration can be written as:
\begin{equation}\label{Nrbm}
     \bar{\textbf{X}}_{T_{k+1}} = \bar{\textbf{X}}_{T_k} + \tau \textbf{b}(\bar{\textbf{X}}_{T_k}) + \tau \textbf{K}^{\xi_k} (\bar{\textbf{X}}_{T_k}) + \sqrt{2\sigma } \left(\textbf{W}_{T_{k+1}}-\textbf{W}_{T_k}\right),
\end{equation}
where 
\begin{equation}\label{b}
    \textbf{b}(\bar{X}_{T_k}) = \left(b\left(\bar{\textbf{X}}_{T_k}^1\right),\cdots,b\left(\bar{X}_{T_k}^N\right)\right)^{T}\in \mathbb{R}^{Nd};
\end{equation}
\begin{equation}\label{K}
    \textbf{K}^{\xi_k} (\bar{\textbf{X}}_{T_k}) = \left(K^{\xi_k}(\bar{X}_{T_k}^1),\cdots,K^{\xi_k}(\bar{X}_{T_k}^N)\right)^T\in \mathbb{R}^{Nd},
\end{equation}
and $\textbf{Z}_k\sim N(0,I_{Nd})$.
At time $t$, the continuous version for $N$ particles at grid point $T_k$ \eqref{onerbm} can be written as:
\begin{equation}\label{disrbm}
    \bar{\textbf{X}}_t = \bar{\textbf{X}}_{T_k} + (t-T_k) \textbf{b}(\bar{\textbf{X}}_{T_k}) + (t-T_k) \textbf{K}^{\xi_k} (\bar{\textbf{X}}_{T_k}) + \sqrt{2\sigma} \left( \textbf{W}_t - \textbf{W}_{T_k}\right),
\end{equation}
where $\textbf{W}_t = \left(W_t^1,\cdots,W_t^N\right)^T$.  

We refer (\ref{Nrbm})  as the discrete-RBM (or fixed batch size RBM), abbreviated as \textbf{d-RBM} and (\ref{disrbm}) as the continuous-RBM, abbreviated as \textbf{c-RBM}. 

Furthermore, denote the $N$ particle joint law of $ \bar{\textbf{X}}_t$ by $\bar{\rho}_t^N$. Following the basic approach introduced in \cite{jabin2016mean,jabin2018quantitative}, our main idea is to use the relative entropy methods to compare the joint law $\bar{\rho}^N_t\left(x_1, \ldots, x_N\right)$ of the c-RBM (\ref{disrbm}) to the tensorized law
$$
\rho^N_t\left(x_1, \ldots, x_N\right)=\rho^{\otimes N}=\Pi_{i=1}^N \rho^i_t\left(x_i\right),
$$
consisting of $N$ independent copies of a process following the law $\rho$, solution to the mean-field limit equation, the Fokker-Planck equation (\ref{FPE}). One can readily check that $\rho_t^N$ solves
\begin{equation}\label{tensorFK}
    \partial_t \rho^N_t + \sum_{i=1}^N \operatorname{div}_{x_i} \left(\rho_t^N \left(b(x_i)+K * \rho_t(x_i)\right)\right)=\sum_{i=1}^N \sigma \Delta_{x_i}\rho_t^N.
    \end{equation}
Our method revolves around the control of the {\it rescaled relative entropy}
$$
\mathcal{H}_N\left( \bar{\rho}^N_t \mid \rho^N_t \right)=\frac{1}{N} \int_{\mathbb{R}^{d N}} \bar{\rho}^N_t\left( \textbf{x}\right) \log \frac{\bar{\rho}^N_t\left(\textbf{x}\right)}{\rho^N_t\left(\textbf{x}\right)} d \textbf{x},
$$
where $\textbf{x}=(x_1,\cdots,x_N)^T$. The proof of our main results is based on the Markov property of the d-RBM. More precisely, let $\bar{\varrho}_t^{N, \boldsymbol{\xi}}$ be the probability density of c-RBM (\ref{disrbm}) for a given sequence of batches $\boldsymbol{\xi}:=\left(\xi_0, \xi_1, \cdots, \xi_k, \cdots\right)$. Consequently, we have the following expression of the density:
$$
\bar{\rho}_t^N(\textbf{x})=\mathbb{E}_{\boldsymbol{\xi}}\left[\bar{\varrho}_t^{N, \boldsymbol{\xi}}(\textbf{x})\right].
$$
Here, $\mathbb{E}_{\boldsymbol{\xi}}\left[\bar{\varrho}_t^{N, \boldsymbol{\xi}}(\cdot)\right]$ means taking expectation for all possible choice of batch $\boldsymbol{\xi}$. Note that $\bar{\varrho}_t^{N, \boldsymbol{\xi}}$ is consistent with an analogue of the Liouville equation, whose explicit expression will be given in Lemma \ref{Lioulemma}. 

\subsection{An analogue of the Liouville equation}
Consider the c-RBM (\ref{disrbm}).
Our goal is to prove that the probability density of the continuous version satisfies an analogue of the Liouville equation with time-varying drift terms.
\begin{lemma}\label{Lioulemma}
 Denote by $\bar{\varrho}_t^{N, \boldsymbol{\xi}}$ the probability density function of $\bar{\textbf{X}}_t = \left(\bar{X}_t^1, \cdots, \bar{X}_t^N\right)$ defined in (\ref{disrbm}) for $t\in\left[T_k, T_{k+1}\right)$. Then the following Liouville equation holds:
\begin{equation}\label{liou}
    \partial_t \bar{\varrho}_t^{N, \boldsymbol{\xi}}+\sum_{i=1}^N \operatorname{div}_{x_i}\left(\bar{\varrho}_t^{N, \boldsymbol{\xi}}\left(\hat{b}_{t}^{\boldsymbol{\xi},i}\left(\textbf{x}\right)+ \hat{K}_t^{\boldsymbol{\xi},i}\left(\textbf{x}\right)\right)\right)=\sum_{i=1}^N \sigma \Delta_{x_i} \bar{\varrho}_t^{N, \boldsymbol{\xi}},
\end{equation}
where
\begin{equation}\label{lioub}
\hat{b}_t^{\boldsymbol{\xi},i} \left(\textbf{x}\right)=\mathbb{E}\left[b\left(\bar{X}_{T_k}^i\right) \mid \bar{\textbf{X}}_t=\textbf{x},\boldsymbol{\xi} \right], \quad t \in\left[T_k, T_{k+1}\right),
\end{equation}
and
\begin{equation}\label{liouK}
\hat{K}_t^{\boldsymbol{\xi},i}\left(\textbf{x}\right):=\mathbb{E}\left[K^{\xi_k}\left(\bar{X}_{T_k}^i\right) \mid \bar{\textbf{X}}_t=\textbf{x},\boldsymbol{\xi}\right], \quad t \in\left[T_k, T_{k+1}\right).
\end{equation}
Here, $\textbf{x}=(x_1, \cdots, x_n) \in \mathbb{R}^{Nd}$.
\end{lemma}

The derivation is not difficult and has appeared in many previous work \cite{li2022sharp,mou2022improved}. For the conveninece of the readers, we also attach a proof in Appendix \ref{liouappend}.

\subsection{Assumptions}
Before the main results and proofs, we firstly give the following assumptions we use throughout this paper.
\begin{assumption}\label{assume}
\begin{itemize}
    \item [(a)]The field $b$: $\mathbb{R}^d \to \mathbb{R}^d$ and the interaction kernel $K$ are Lipshitz:
\[
\left|b(x_1)-b(x_2)\right| \leq r \left|x_1-x_2\right|, \quad \left|K(x)-K(y)\right| \leq L|x-y|.
\]
\item[(b)] The field $b$ is strongly confining in the sense that there exists two constants $\alpha$ and $\beta$ such that for any $x_1 \neq x_2$, :
$$(x_1 -x_2) \cdot(b(x_1)-b(x_2)) \leq \alpha -\beta|x_1-x_2|^2$$
for some constant $\beta>0$ and  $\beta>2L$.
   \item [(c)] Moreover, the Hessians of $b$ and $K$ have at most polynomial growth, namely,
   $$\left|\nabla^2 b(x)\right|  \leq  C(1 + \left|x\right| )^q, \quad \left|\nabla^2 K(x)\right|  \leq  \tilde{C} (1 + \left|x\right| )^q.$$
   \item [(d)]  $K^\xi - F$ is uniformly bounded, namely,
$$
\operatorname{esssup}_{\xi} \|K^{\xi} - F \|_{L^{\infty}(\mathbb{R}^d)} < +\infty,
$$
where $F(x_i) = \frac{1}{N}\sum_{j=1}^N K(x_i - x_j)$.
   \end{itemize}
\end{assumption}

The confining property of $b$ holds if there exists a compact set $S \subset \mathbb{R}^d$ such that for any $x_1, x_2\notin S$, it holds that:
$$
(x_1 -x_2) \cdot(b(x_1)-b(x_2)) \leq -\beta|x_1-x_2|^2.
$$
By definition, $K^{\xi}$ satisfies the same conditions as those for $K$ for any $\xi$.

\section{Some auxiliary results}\label{sec3}

\subsection{The Log-Sobolev inequality}
In this paper, we obtain the  uniform in time estimate under an additional assumption of a log-Sobolev inequality (LSI) for $\rho_t$, uniformly in $t$. We remark that such assumption is a common ingredient in the proof of uniform-in-time propagation of chaos, see also the recent paper \cite{lacker2023sharp}.

\begin{assumption}
    The Log-Sobolev inequality (LSI): There exists a constant $C_{LS}>0$ such that for any nonnegative smooth functions $f$, one has
    \begin{equation}\label{LSI}
        \operatorname{Ent}_{\rho_t}(f):=\int f \log f d\rho_t - \left(\int f d\rho_t \right) \log \left(\int f d\rho_t \right) \leq C_{LS} \int
    \frac{\left|\nabla f\right|^2}{f} d\rho_t.
    \end{equation}
\end{assumption}\

One crucial property of the LSI is the tensorization, i.e. if $\rho_t$ satisfies a LSI then $\rho_t^N=\rho_t^{\otimes^N}$ satisfies the same inequality with the same constant (and thus independent of $N$). 


\subsection{Moment control}
In this section, we aim to find a uniform-in-time bound for the moments $\mathbb{E}\left|\bar{X}_t^i\right|^p$. We have the following fact. The proof is similar to \cite{jin2022mean}, so we omit it.

\begin{lemma}\label{moment}
    Under Assumption \ref{assume}, for any $p>2$, there exists a constant $C_p$ independent of $N$ and $\xi$ such that for any $i$
    \begin{equation}
        \begin{aligned}
             \sup_{t\geq 0} \mathbb{E}\left[\left|\bar{X}_t^i\right|^p\mid \mathcal{F}_{\xi}\right] \leq C_p,
        \end{aligned}
    \end{equation}
    where $\mathcal{F}_{\xi}$ denotes the $\sigma$-algebra generated by sequence $\xi$.
\end{lemma}
The above result indicates that for a fixed sequence of divisions of random batches $\xi=(\xi_k)_{k\geq0}$, the $p$th moment of $\bar{X}_t^i$ is bounded by $C_p$ independent of $\xi$.
We remark that using the moment control, we can bound terms like $b(\bar{X}_t^i)$ and $K^{\xi}(\bar{X}_t^i)$ which can be found in the proof Theorem \ref{thm}.

\subsection{Estimate of the Fisher information}

The Fisher information for a probability measure $\rho$ is defined by

\begin{equation}
    \mathcal{I}(\rho) = \int \left|\nabla \log \rho\right|^2 \rho dx.
\end{equation}
In our analysis, we require a bound for the Fisher information of $\bar{\varrho}_{T_k}^{N, \boldsymbol{\xi}}$, which is the law of RBM for a given sequence of batches $\boldsymbol{\xi} = (\xi_0, \cdots, \xi_k, \cdots)$ at grid point $T_k$. Our proof is based on Stam's convolution inequality for Fisher information \cite[Eq. (2.9)]{stam1959some}. This inequality guarantees that for any pair of suitably regular probability density functions $p$, $q$ on $\mathbb{R}^d$, the Fisher information satisfies the inequality
\begin{equation}\label{fisherineq}
    \frac{1}{\mathcal{I}(p*q)}\geq \frac{1}{\mathcal{I}(p)}+\frac{1}{\mathcal{I}(q)},
\end{equation}
where $p*q$ denotes the convolution of $p$ and $q$. The d-RBM (\ref{Nrbm}) for $N$ particles at time $T_{k+1}$ can be seen as a combination of applying the deterministic mapping 
\begin{equation}
    \psi_{\tau}^{\xi_k}(\textbf{x}) := \textbf{x} + \tau \left(\textbf{b}(\textbf{x})+\textbf{K}^{\xi_k}(\textbf{x})\right), \quad \textbf{x} = \left(x_1,\cdots,x_n\right)\in \mathbb{R}^{Nd},
\end{equation}
with a convolution step with a Gaussian kernel, where $\textbf{b}$ and $\textbf{K}^{\xi_k}$
are defined as (\ref{b}) and (\ref{K}) respectively. More precisely, $$\textbf{b}(\textbf{x}) = \left(b(x_1), \cdots, b(x_n)\right)^T \in \mathbb{R}^{Nd},$$
and $$\textbf{K}^{\xi_k}(\textbf{x})=\left(K^{\xi_k}(x_1), \cdots, K^{\xi_k}(x_n)\right)^T \in \mathbb{R}^{Nd}.$$
We exploit the inequality (\ref{fisherineq}) so as to bound the Fisher information $\mathcal{I}(\bar{\varrho}_{T_{k+1}}^{N, \boldsymbol{\xi}})$ in terms of $\mathcal{I}(\bar{\varrho}_{T_k}^{N, \boldsymbol{\xi}})$. In order to do so, we bound the Fisher information for the intermediate density after the first step.
\begin{lemma}
    For some step size $\tau < \frac{1}{2(r+L)} $ , let $p_k(\cdot)$ be the density of the random variable $\textbf{Z}_k = \psi_{\tau}^{\xi_k}(\bar{\textbf{X}}_{T_k})$ obtained by applying the deterministic mapping $\psi_{\tau}^{\xi_k}$. Then under Assumption \ref{assume}, we have the bound
\begin{equation}
    \mathcal{I}(p_k) \leq \frac{1+\tau(r+L)}{1-\tau(r+L)} \left( \mathcal{I}(\bar{\varrho}_{T_k}^{N, \boldsymbol{\xi}}) + \frac{M(r+L) N\tau}{1-\tau (r+L)} \right),
\end{equation}
where $M$ is a constant independent of $N$.
\end{lemma}
\begin{proof}
Consider the norm of $\textbf{x}$ defined as $\|\textbf{x}\|^2 = \sum_i^n \|x_i\|^2_2$. Then, Assumption \ref{assume} implies that for $\tau < \frac{1}{r+L}$, the mapping $\psi_{\tau}^{\xi_k}(\textbf{x}) := \textbf{x} + \tau \left(\textbf{b}(\textbf{x})+\textbf{K}^{\xi_k}(\textbf{x})\right)$ is a bi-Lipschitz mapping, i.e.
    \begin{equation*}
        (1-\tau(r+L))\|\textbf{x} - \textbf{y}\| \leq \|\psi_{\tau}^{\xi_k}(\textbf{x})-\psi_{\tau}^{\xi_k}(\textbf{y}) \| \leq (1+\tau(r+L)) \|\textbf{x} - \textbf{y}\|,
    \end{equation*}
from which one deduces that
\begin{equation}\label{suppsi}
    \|\nabla \psi_{\tau}^{\xi_k}(\textbf{x})\| \leq (1+\tau(r+L)), \quad \|\nabla \psi_{\tau}^{\xi_k}(\textbf{x})^{-1}\|\le (1-\tau(r+L))^{-1}.
\end{equation}
By the change of variable formula, we have
$$
p_k(\textbf{z}) = \frac{\bar{\varrho}_{T_k}^{N, \boldsymbol{\xi}}(\textbf{x})}{\det (\nabla \psi_{\tau}^{\xi_k}(\textbf{x}))}.
$$
Consequently, we have the bound on the Fisher information:
$$
\begin{aligned}
     \mathcal{I}(p_k) = & \int_{\mathbb{R}^{Nd}} p_k(\textbf{z}) \left|\nabla_{\textbf{z}}\log p_k(\textbf{z})\right|^2 d\textbf{z} \\
     = & \int_{\mathbb{R}^{Nd}} \bar{\varrho}_{T_k}^{N, \boldsymbol{\xi}}(\textbf{x}) \left| \nabla \psi_{\tau}^{\xi_k}(\textbf{x})^{-1} \left(\nabla_{\textbf{x}}\log \bar{\varrho}_{T_k}^{N, \boldsymbol{\xi}}(\textbf{x})-\nabla_{\textbf{x}}\log\det \left(\nabla \psi_{\tau}^{\xi_k}(\textbf{x}) \right) \right)\right|^2 d\textbf{x}\\
     \leq &  (1+\tau (r+L))\int_{\mathbb{R}^{Nd}} \bar{\varrho}_{T_k}^{N, \boldsymbol{\xi}}(\textbf{x}) \left|\nabla \psi_{\tau}^{\xi_k}(\textbf{x})^{-1} \nabla_{\textbf{x}}\log \bar{\varrho}_{T_k}^{N, \boldsymbol{\xi}}(\textbf{x})\right|^2 d\textbf{x} \\
     & +  \left(1+\frac{1}{\tau (r+L)}\right)\int_{\mathbb{R}^{Nd}} \bar{\varrho}_{T_k}^{N, \boldsymbol{\xi}}(\textbf{x}) \left|\nabla \psi_{\tau}^{\xi_k}(\textbf{x})^{-1} \nabla_{\textbf{x}}\log\det \left(\nabla \psi_{\tau}^{\xi_k}(\textbf{x}) \right) \right|^2 d\textbf{x}=:I_1+I_2.
\end{aligned}
$$
The inequality follows from the simple fact $(a+b)^2\le (1+\lambda) a^2 + (1+\frac{1}{\lambda}) b^2$. It is clear that
\[
I_1\le \frac{1+\tau(r+L)}{1-\tau(r+L)}
\int_{\mathbb{R}^{Nd}} \bar{\varrho}_{T_k}^{N, \boldsymbol{\xi}}(\textbf{x}) \left|  \nabla\log \bar{\varrho}_{T_k}^{N, \boldsymbol{\xi}}(\textbf{x})\right|^2 d\textbf{x}
=\frac{1+\tau(r+L)}{1-\tau(r+L)}
\mathcal{I}(\bar{\varrho}_{T_k}^{N, \boldsymbol{\xi}}).
\]
The second term can be estimated by
\[
\begin{aligned}
I_2 \leq &  \frac{2}{\tau (r+L)}\int_{\mathbb{R}^{Nd}}\bar{\varrho}_{T_k}^{N, \boldsymbol{\xi}}(\textbf{x})  \left|\nabla \psi_{\tau}^{\xi_k}(\textbf{x})^{-2} \nabla \cdot  \left(I_{Nd} + \tau\nabla \textbf{b}(\textbf{x}) + \tau\nabla\textbf{K}^{\xi_k}(\textbf{x}) \right) \right|^2 d\textbf{x}\\
     \leq &  
     \frac{(1+\tau(r+L))\tau}{(1-\tau(r+L))^2(r+L)}\int_{\mathbb{R}^{Nd}} \bar{\varrho}_{T_k}^{N, \boldsymbol{\xi}}(\textbf{x}) \left| \nabla \cdot  \left( \nabla \textbf{b}(\textbf{x}) + \nabla \textbf{K}^{\xi_k}(\textbf{x}) \right) \right|^2 d\textbf{x}\\
     \leq &  \frac{(1+\tau(r+L))N\tau}{(1-\tau(r+L))^2}\frac{M}{r+L},
\end{aligned}
\]
where the last inequality above is due to the polynomial growth of the Hessians and the boundness of the moments, and $M$ is a constant independent of $N$.
\end{proof}

Let $q_{\tau}$ denote the $Nd$-dimensional Gaussian distribution $\mathcal{N}(0,2\sigma\tau \textbf{I}_{Nd})$. Clearly we have the identity $\mathcal{I}(q_\tau) = \frac{Nd}{2\sigma\tau}$. By the d-RBM  (\ref{Nrbm}), we have that $ \bar{\varrho}_{T_{k+1}}^{N, \boldsymbol{\xi}}= p_k(\textbf{x})*q_{\tau}$,  Invoking the convolution inequality (\ref{fisherineq}), for $\tau < \frac{1}{2(r+L)}$ , we have the bound
\begin{equation}\label{recursion}
    \frac{1}{\mathcal{I}\left(\bar{\varrho}_{T_{k+1}}^{N, \boldsymbol{\xi}}\right)} \geq \frac{1}{\mathcal{I}\left(p_k\right)}+\frac{1}{\mathcal{I}\left(q_\tau\right)} \geq \frac{(1-\tau(r+L))^2}{1+\tau(r+L)}\frac{1}{\max(\mathcal{I}\left(\bar{\varrho}_{T_{k}}^{N, \boldsymbol{\xi}}\right), M N)}+\frac{2\sigma\tau}{Nd}.
\end{equation}

Let $u_k= \frac{1}{\mathcal{I}\left(\bar{\varrho}_{T_{k}}^{N, \boldsymbol{\xi}}\right)}$. Then, one has a relation like
\[
u_{k+1}\ge \min(\gamma u_k+\delta, \gamma B+\delta) 
\]
where $\gamma=(1-\tau(r+L))^2/(1+\tau(r+L))$, $B=\frac{1}{MN}$ and $\delta=\frac{2\sigma\tau}{Nd}$. For this recursion, one easily obtains
\[
u_k\ge \min\left(u_0, \gamma B+\delta, \frac{\delta}{1-\gamma}\right),
\]
because $u<\gamma u+\delta$ for $u<\delta/(1-\gamma)$ and $u>\gamma u+\delta$  for $u>\delta/(1-\gamma)$.

Consequently, one can obtain the following control for the Fisher information. The upper bound clearly scales linearly with $N$.
\begin{lemma}\label{fishlemma}
    Under Assumption \ref{assume}, we have the following bound of the Fisher information independent of the batch $\boldsymbol{\xi}=(\xi_0, \cdots, \xi_k, \cdots)$:
\begin{equation}
\mathcal{I}\left(\bar{\varrho}_{T_{k}}^{N, \boldsymbol{\xi}}\right) \leq \max\left(\mathcal{I}\left(\rho^N_{0}\right), \frac{1+\tau(r+L)}{(1-\tau(r+L))^2}MN ,\frac{Nd(r+L)(3+\tau(r+L))}{2\sigma}
\right).
\end{equation}
\end{lemma}

\section{The main results}\label{mainthm}

Equipped with the preparation work before, we now establish an $O(\tau^2)$ error estimate for RBM in terms of the scaled relative entropy. Firstly, we notice that the d-RBM (\ref{Nrbm}) at discrete time points is a time homogeneous Markov chain, and $\bar{\rho}_{T_k}^N$ is the law at $T_k$. Recall that $\bar{\varrho}_t^{N,\boldsymbol{\xi}}$ is the probability density of the c-RBM (\ref{disrbm}) for a given sequence of batches $\boldsymbol{\xi}:=\left(\xi_0, \xi_1, \cdots, \xi_k, \cdots\right)$, so that $\bar{\rho}_{T_k}^N=\mathbb{E}_{\boldsymbol{\xi}}\left[\bar{\varrho}_{T_k}^{N, \boldsymbol{\xi}}\right]$. Moreover, by the Markov property, we are able to define
\begin{equation}\label{markovp}
\bar{\rho}_t^{N, \xi_k}:=\mathbb{E}\left[\bar{\varrho}_t^{N, \boldsymbol{\xi}} \mid \xi_i, i \geq k\right]=\mathcal{S}_{T_k, t}^{N, \xi_k} \bar{\rho}_{T_k}^N, \quad t \in\left[T_k, T_{k+1}\right),
\end{equation}
where the operator $\mathcal{S}_{T_k, t}^{N, \xi_k}$ is the evolution operator from $T_k$ to $t$ for the Liouville equation of the c-RBM (\ref{disrbm}) derived in Lemma \ref{Lioulemma}, for some given $\xi_k$ :
\begin{equation}\label{proofliou}
    \partial_t \bar{\rho}^{N, \xi_k}_t+\sum_{i=1}^N \operatorname{div}_{x_i}\left(\bar{\rho}^{N, \xi_k}_t\left(\bar{b}_t^{\xi_k, i}\left(\textbf{x}\right)+ \bar{K}^{\xi_k, i}_t\left(\textbf{x}\right)\right)\right)=\sum_{i=1}^N \sigma \Delta_{x_i} \bar{\rho}^{N, \xi_k}_t, \quad \bar{\rho}^{N, \xi_k}_{T_k} = \bar{\rho}^{N}_{T_k},
\end{equation}
where
$$
\bar{b}_t^{\xi_k,i}\left(\textbf{x}\right):=\mathbb{E}\left[b\left(\bar{X}_{T_k}^i\right) \mid \bar{\textbf{X}}_t=\textbf{x}, \xi_k\right], \quad t \in\left[T_k, T_{k+1}\right),
$$
and
$$
\bar{K}_t^{\xi_k,i}\left(\textbf{x}\right):=\mathbb{E}\left[K^{\xi_k}\left(\bar{X}_{T_k}^i\right) \mid \bar{\textbf{X}}_t=\textbf{x}, \xi_k\right], \quad t \in\left[T_k, T_{k+1}\right).
$$
Here, $\bar{b}_t^{\xi_k,i}\left(\textbf{x}\right)$ and $\bar{K}_t^{\xi_k,i}\left(\textbf{x}\right)$,  the expectation conditional on the $k$-th batch $\xi_k$,  are a little bit different from (\ref{lioub}) and (\ref{liouK}) which are the expectation conditional on all fixed batch $\boldsymbol{\xi}=(\xi_0, \cdots, \xi_k, \cdots)$.

Since for $t\in \left[ T_k,T_{k+1}\right)$,
$$\bar{\rho}_{t}^N=\mathbb{E}_{\xi_k}\left[\bar{\rho}_{t}^{N,\xi_k}\right],$$
then by (\ref{proofliou}) one gets
\begin{equation}\label{noxiliou}
    \partial_t \bar{\rho}^{N}_t = -\sum_{i=1}^N \mathbb{E}_{\xi_k}\left[\operatorname{div}_{x_i}\left(\bar{\rho}^{N, \xi_k}_t\left(\bar{b}_t^{\xi_k,i}\left(\textbf{x}\right)+ \bar{K}^{\xi_k,i}_t\left(\textbf{x}\right)\right)\right)+ \sigma \Delta_{x_i} \bar{\rho}^{N, \xi_k}_t\right].
\end{equation}

\begin{theorem}\label{thm}
Consider the joint law $\bar{\rho}^N_t$ for $N$ particle $\bar{\rho}_{t}^{N}$ defined in (\ref{disrbm}) and the tensorized law $\rho^N_t$ of the Fokker-Plank equation (\ref{FPE}). Suppose $\sigma >0 $, then under Assumptions \ref{assume} and \ref{LSI}, there exist positive constants $c_1$, $c_2$ and $\Delta$ independent of $N$ and $t$ but dependent on $\sigma$, such that for all $\tau \in (0, \Delta)$, 
    \begin{equation}
        \sup_{t} \mathcal{H}_N\left( \bar{\rho}^N_t \mid \rho^N_t \right) \leq \mathcal{H}_N\left( \bar{\rho}^N_0 \mid \rho^N_0 \right)+c_1 \tau^2 + \frac{c_2}{N}.
    \end{equation}
\end{theorem}

It is well-known that one can bound the Wasserstein-1 distance with square root
of the relative entropy by the Csisz{\'a}r-Kullback-Pinsker inequality, see for instance \cite{villani2009optimal}.

\begin{corollary}
    Under the same assumptions as in Theorem \ref{mainthm}, the following hold:
    \begin{equation}
    W_1\left(\bar{\rho}_{t}^{N}, \rho_{t}^{N}\right) \leq C_1 \tau + \frac{C_2}{\sqrt{N}}.
    \end{equation}
    Here the constants $C_1, C_2$ are independent of $N$ and $t$.
\end{corollary}
\begin{proof}[Proof of Theorem \ref{mainthm}]
We prove this result in six main steps.\\

\textbf{Step 1:} Estimate time derivative of the relative entropy.\\

By using the equations (\ref{tensorFK}) and (\ref{noxiliou}) for $\rho^N_t$ and $\bar{\rho}_t^N$, respectively, we calculate the time derivative of the relative entropy in the time interval $\left[T_k,T_{k+1}\right)$:
\begin{equation}
\begin{aligned}
& \frac{d}{d t}   \mathcal{H}_N\left( \bar{\rho}^N_t \mid \rho^N_t \right)=\frac{1}{N}\int_{\mathbb{R}^{Nd}}\left(\partial_t \bar{\rho}_t^N\right)\left(\log \frac{\bar{\rho}_t^N}{\rho_t^N}+1\right) d \textbf{x}+\frac{1}{N}\int_{\mathbb{R}^{Nd}}\left(\partial_t \rho_t^N\right)\left(-\frac{\bar{\rho}_t^N}{\rho_t^N}\right) d \textbf{x} \\
= & \frac{1}{N}\sum_{i=1}^N \int_{\mathbb{R}^{Nd}} \left(\mathbb{E}_{\xi_k}\left(\bar{\rho}_t^{N,\xi_k}\left(\bar{b}_t^{\xi_k,i}(\textbf{x})+\bar{K}_t^{\xi_k,i}(\textbf{x})\right)\right)-\sigma \operatorname{div}_{x_i} \bar{\rho}_t^N \right) \cdot\left(\nabla_{x_i} \log \frac{\bar{\rho}_t^N}{\rho_t^N }\right) d\textbf{x}\\
& + \frac{1}{N}\sum_{i=1}^N \int_{\mathbb{R}^{Nd}} \left( \rho_t^N \left(b(x_i)+K * \rho_t(x_i)\right)-\sigma \operatorname{div}_{x_i} \rho_t^N \right) \cdot\left(-\nabla_{x_i} \frac{\bar{\rho}_t^N}{\rho_t^N}\right) d \textbf{x}.
\end{aligned}
\end{equation}
Introduce 
\begin{equation}\label{Fxi}
    F(x_i) = \frac{1}{N-1}\sum_{j:j\neq i}K(x_i-x_j),
\end{equation}
rearrange the terms
\begin{equation}
    \begin{aligned}
        & \frac{d}{d t}   \mathcal{H}_N\left( \bar{\rho}^N_t \mid \rho^N_t \right)=\frac{1}{N}\sum_{i=1}^N \int_{\mathbb{R}^{Nd}}\mathbb{E}_{\xi_k}\left( \rho_t^{N,\xi_k}\left( \bar{b}_t^{\xi_k,i}(\textbf{x})-b(x_i) \right)\right)\cdot \nabla_{x_i} \log \frac{\bar{\rho}_t^N}{\rho_t^N} d\textbf{x}\\
        & + \frac{1}{N}\sum_{i=1}^N \int_{\mathbb{R}^{Nd}} \mathbb{E}_{\xi_k} \left( \bar{\rho}^{N,\xi_k}_t \bar{K}^{\xi_k,i}_t(\textbf{x}) - \bar{\rho}^{N,\xi_k}_t K^{\xi_k}(x_i)\right) \cdot \nabla_{x_i} \log \frac{\bar{\rho}_t^N}{\rho_t^N } d \textbf{x}\\
        & + \frac{1}{N}\sum_{i=1}^N \int_{\mathbb{R}^{Nd}} \mathbb{E}_{\xi_k}\left( \bar{\rho}^{N,\xi_k}_t K^{\xi_k}(x_i)-\bar{\rho}^N_t F(x_i)\right) \cdot \nabla_{x_i} \log \frac{\bar{\rho}_t^N}{\rho_t^N } d \textbf{x}\\
        & + \frac{1}{N}\sum_{i=1}^N \int_{\mathbb{R}^{Nd}} \left( F(x_i)-K*\rho_t(x_i)\right)\bar{\rho}^N_t \cdot \nabla_{x_i} \log \frac{\bar{\rho}_t^N}{\rho_t^N } d \textbf{x}\\
        & -\frac{\sigma}{N} \sum_{i=1}^N \int_{\mathbb{R}^{Nd}} \bar{\rho}^N_t \left|\nabla_{x_i} \log \frac{\bar{\rho}_t^N}{\rho_t^N }\right|^2 d \textbf{x}\\
        & := \frac{1}{N}\sum_{i=1}^N (J_1^i+J_2^i+J_3^i+J_4^i+J_5^i).
    \end{aligned}
\end{equation}
As done in \cite{li2022sharp}, for $J_3^i$, our approach is to introduce another copy of RBM $\bar{\textbf{Y}}$ that depends on another batch $\tilde{\xi_k}$ such that
\begin{itemize}
    \item $\bar{\textbf{Y}}_{T_k}=\bar{\textbf{X}}_{T_k}$;
    \item the Brownian Motion are the same in $\left[T_k, T_{k+1}\right)$;
    \item the batch $\tilde{\xi}_k$ on $\left[T_k, T_{k+1}\right)$ is independent of $\xi_k$.
\end{itemize}
Consequently, density of the law $\bar{\rho}_t^{N,\tilde{\xi}_k}$ for $\bar{\textbf{Y}}$ satisfies both (\ref{markovp}) and (\ref{proofliou}). Then 
\begin{equation}\label{j3}
    \begin{aligned}
        J_3^i  
         &=  \int_{\mathbb{R}^{Nd}} \mathbb{E}_{\xi_k} \left[ \left( K^{\xi_k}(x_i)- F(x_i) \right)\left(\bar{\rho}^{N,\xi_k}_t-\bar{\rho}^{N}_t \right)\right] \cdot \nabla_{x_i} \log \frac{\bar{\rho}_t^N}{\rho_t^N }d \textbf{x}\\
         &=  \int_{\mathbb{R}^{Nd}} \mathbb{E}_{\xi_k,\tilde{\xi}_k} \left[ \left( K^{\xi_k}(x_i)- F(x_i) \right)\left(\bar{\rho}^{N,\xi_k}_t-\bar{\rho}^{N,\tilde{\xi}_k}_t \right)\right]\cdot \nabla_{x_i} \log \frac{\bar{\rho}_t^N}{\rho_t^N }d \textbf{x}\\
         & \leq \frac{2}{\sigma} \mathbb{E}_{\xi_k, \tilde{\xi}_k} \left[ \int_{\mathbb{R}^{Nd}} \left| K^{\xi_k}(x_i)- F(x_i)\right|^2 \frac{\left|\bar{\rho}^{N,\xi_k}_t-\bar{\rho}^{N,\tilde{\xi}_k}_t\right|^2}{\bar{\rho}^{N,\tilde{\xi}_k}_t}d\textbf{x}\right] +\frac{\sigma}{8} \int_{\mathbb{R}^{Nd}} \bar{\rho}^N_t \left|\nabla_{x_i} \log \frac{\bar{\rho}_t^N}{\rho_t^N }\right|^2 d \textbf{x}.
    \end{aligned}
\end{equation}
Here the second equality is due to the fact $$\mathbb{E}_{\xi_k}\left[\bar{\rho}^{N,\xi_k}_t\right] = \bar{\rho}^{N}_t,$$ and the consistency of the random batch, that is $$\mathbb{E}_{\xi_k}\left[K^{\xi_k}(x_i)\right] = F(x_i).$$
Then the term $\left| \bar{\rho}^{N,\xi_k}_t K^{\xi_k}(x_i)-\bar{\rho}^N_t F(x_i)\right|$ is of $O(1)$. The second equality is due to $\mathbb{E}_{\tilde{\xi}_k}\left[\bar{\rho}^{N,\tilde{\xi}_k}_t \right]=\bar{\rho}^{N}_t$. In the last inequality, we applied Young's inequality and the fact $\mathbb{E}_{\xi_k, \tilde{\xi}_k}\left|\nabla_{x_i} \log \frac{\bar{\rho}_t^N}{\rho_t^N}\right|^2 \bar{\rho}_t^{N,\tilde{\xi}_k}=\left|\nabla_{x_i} \log \frac{\bar{\rho}_t^N}{\rho_t^N}\right|^2 \bar{\rho}_t^N$. The introduction of the independent copy of $\tilde{\xi}_k$ is useful since we may apply the Girsanov transform later to estimate this quantitatively.

Then apply Young's inequality for $J_1^i$, $J_2^i$, $J_4^i$ and the Log-Sobolev inequality for $J_3^i$, we have
\begin{equation}
    \begin{aligned}
          \frac{d}{d t}   \mathcal{H}_N\left( \bar{\rho}^N_t \mid \rho^N_t \right) \leq & -\frac{\sigma}{2 C_{LS}} \mathcal{H}_N\left( \bar{\rho}^N_t \mid \rho^N_t \right) + \frac{2}{\sigma N}\sum_{i=1}^N\mathbb{E}_{\xi_k} \int_{\mathbb{R}^{Nd}} \left|\bar{b}^{\xi_k,i}_t(\textbf{x}) - b(x_i)\right|^2 \bar{\rho}^{N,\xi_k}_t d \textbf{x}\\
           & + \frac{2}{\sigma N}\sum_{i=1}^N \mathbb{E}_{\xi_k} \int_{\mathbb{R}^{Nd}} \left| \bar{K}^{\xi_k,i}_t(\textbf{x})-K^{\xi_k}(x_i) \right|^2 \bar{\rho}^{N,\xi_k}_t d \textbf{x} \\
           & + \frac{2}{\sigma N}\sum_{i=1}^N \mathbb{E}_{\xi_k, \tilde{\xi}_k} \left[ \int_{\mathbb{R}^{Nd}} \left| K^{\xi_k}(x_i)- K(x_i)\right|^2 \frac{\left|\bar{\rho}^{N,\xi_k}_t-\bar{\rho}^{N,\tilde{\xi}_k}_t\right|^2}{\bar{\rho}^{N,\tilde{\xi}_k}_t}\right]d\textbf{x}\\
           & + \frac{2}{\sigma N}\sum_{i=1}^N \int_{\mathbb{R}^{Nd}} \bar{\rho}_t^N \left|F(x_i)-K*\rho_t(x_i)\right|^2 d \textbf{x}.
    \end{aligned}
\end{equation}

\textbf{Step 2:} Estimate $\mathbb{E}_{\xi_k}\int_{\mathbb{R}^{Nd}} \left|\bar{b}_t^{\xi_k,i}(\textbf{x}) - b(x_i)\right|^2 \bar{\rho}^{N,\xi_k}_t d \textbf{x} $.\\

For any $ t\in\left[T_k,T_{k+1}\right)$, by Taylor's expansion, we have
\begin{equation*}
    \begin{aligned}
        \bar{b}_t^{\xi_k,i}(\textbf{x}) - b(x_i) =& \mathbb{E} \left[b\left(\bar{X}^i_{T_k}\right)-b\left(\bar{X}^i_{t}\right)\mid\bar{\textbf{X}}_t = \textbf{x}, \xi_k\right]\\
        = & \mathbb{E} \left[\bar{X}^i_{T_k}-\bar{X}^i_{t}\mid\bar{\textbf{X}}_t = \textbf{x}, \xi_k\right]\cdot \nabla_{x_i} b(x_i)+\bar{r}_t(x_i),
    \end{aligned}
\end{equation*}
where the remainder takes the form
\begin{equation}
   \bar{r}_t(x_i):=\frac{1}{2}\mathbb{E}\left[ \left(\bar{X}^i_{T_k}-\bar{X}^i_{t}\right)^{\otimes 2}: \int_0^1 \nabla_{x_i}^2 b\left((1-s) \bar{X}_t^i+s \bar{X}_{T_k}^i\right) d s\mid \bar{\textbf{X}}_t = \textbf{x}, \xi_k\right].
\end{equation}
In Lemma \ref{reminder}, we show that for $t \in\left[T_k, T_{k+1}\right)$,
$$
\mathbb{E}\left[\left|\bar{r}_t\left(\bar{X}_t^i\right)\right|^2 \mid \xi_k \right] \leq c \left(t-T_k\right)^2,
$$
where $c$ is a positive constant independent of $N$ and $\xi_k$. 

For the first term, consider 
$\mathbb{E} \left[\bar{\textbf{X}}_{T_k}-\bar{\textbf{X}}_{t}\mid\bar{\textbf{X}}_t = \textbf{x}, \xi_k\right]$. By Bayes' law,
\begin{equation}\label{step2bayes}
    \begin{aligned}
\mathbb{E} \left[\bar{\textbf{X}}_{T_k}-\bar{\textbf{X}}_{t}\mid\bar{\textbf{X}}_t = \textbf{x}, \xi_k\right] & =\int_{\mathbb{R}^{Nd}}(\textbf{y}-\textbf{x}) P\left(\bar{\textbf{X}}_{T_k}=\textbf{y} \mid \bar{\textbf{X}}_t=\textbf{x},\xi_k\right) d \textbf{y} \\
& =\int_{\mathbb{R}^{Nd}}(\textbf{y}-\textbf{x}) \frac{\bar{\rho}_{T_k}^{N}(\textbf{y}) P\left(\bar{\textbf{X}}_t=\textbf{x} \mid \bar{\textbf{X}}_{T_k}=\textbf{y},\xi_k \right)}{\bar{\rho}_t^{N,\xi_k}(\textbf{x})} d \textbf{y}.
\end{aligned}
\end{equation}
Clearly, the distribution $P\left(\bar{\textbf{X}}_t=\textbf{x} \mid \bar{\textbf{X}}_{T_k}=\textbf{y},\xi_k \right)$ is Gaussian, namely,
$$
P\left(\bar{\textbf{X}}_t=\textbf{x} \mid \bar{\textbf{X}}_{T_k}=\textbf{y},\xi_k \right)=\left(4 \pi \sigma\left(t-T_k\right)\right)^{-\frac{Nd}{2}} \exp \left(-\frac{\left|\textbf{x}-\textbf{y}-\textbf{b}(\textbf{y})\left(t-T_k\right)-\textbf{K}^{\xi_k}(\textbf{y})\left(t-T_k\right)\right|^2}{4 \sigma \left(t-T_k\right)}\right).
$$
Then one can use integration by parts to calculate (\ref{step2bayes}). Indeed, we show in Lemma \ref{remainb} that
\begin{equation}
    \int_{\mathbb{R}^{Nd}} \left|\mathbb{E} \left[\bar{\textbf{X}}_{T_k}-\bar{\textbf{X}}_{t}\mid\bar{\textbf{X}}_t = \textbf{x},\xi_k\right]\right|^2 \bar{\rho}^{N,\xi_k}_t d \textbf{x} \leq c\tau^2\left(3N+\mathcal{I}\left(\bar{\rho}_{T_k}^{N}\right)\right).
\end{equation}
Hence, 
\begin{equation}
\begin{aligned}
    &\frac{1}{N}\sum_{i=1}^N \mathbb{E}_{\xi_k} \int_{\mathbb{R}^{Nd}} \left|\bar{b}_t^{\xi_k,i}(\textbf{x}) - b(x_i)\right|^2 \bar{\rho}^{N,\xi_k}_t d \textbf{x} \\
    \lesssim  & \mathbb{E}_{\xi_k}\left[\frac{1}{N}\sum_{i=1}^N \mathbb{E}\left[\left|\bar{r}_t\left(\bar{X}_t^i\right)\right|^2 \mid \xi_k\right] +  \frac{1}{N}\sum_{i=1}^N \int_{\mathbb{R}^{Nd}}\left|\mathbb{E} \left[\bar{X}^i_{T_k}-\bar{X}^i_{t}\mid\bar{\textbf{X}}_t = \textbf{x},\xi_k\right]\right|^2 \bar{\rho}^{N,\xi_k}_t d\textbf{x} \right]\\
     \leq  &\mathbb{E}_{\xi_k}\left[ \frac{1}{N}\sum_{i=1}^N \mathbb{E}\left[\left|\bar{r}_t\left(\bar{X}_t^i\right)\right|^2\mid \xi_k \right] +  \frac{1}{N} \int_{\mathbb{R}^{Nd}} \left|\mathbb{E} \left[\bar{\textbf{X}}_{T_k}-\bar{\textbf{X}}_{t}\mid\bar{\textbf{X}}_t = \textbf{x},\xi_k\right]\right|^2 \bar{\rho}^{N,\xi_k}_t d \textbf{x}\right] \\
     \leq & \tilde{c} \tau^2
   \left(3+\frac{1}{N}\mathcal{I}\left(\bar{\rho}_{T_k}^{N}\right)\right),
\end{aligned}
\end{equation}
where the positive constant $\tilde{c}$ is independent of $\xi_k$, $N$ and $t$, and this is what we desire. 

\textbf{Step 3:} Estimate $\mathbb{E}_{\xi_k} \int_{\mathbb{R}^{Nd}} \left| \bar{K}_t^{\xi_k,i}(\textbf{x})-K^{\xi_k}(x_i) \right|^2 \bar{\rho}^{N,\xi_k}_t d \textbf{x}$.\\
Similar to step 2, by Taylor's expansion, for any $t\in\left[T_k,T_{k+1}\right)$, 
\begin{equation*}
    \begin{aligned}
        \bar{K}^{\xi_k,i}_t(\textbf{x})-K^{\xi_k}(x_i) =& \mathbb{E} \left[K^{\xi_k}\left(\bar{X}^i_{T_k}\right)-K^{\xi_k}\left(\bar{X}^i_{t}\right)\mid\bar{\textbf{X}}_t = \textbf{x}, \xi_k\right]\\
        = & \mathbb{E} \left[\bar{X}^i_{T_k}-\bar{X}^i_{t}\mid\bar{\textbf{X}}_t = \textbf{x}, \xi_k\right]\cdot \nabla_{x_i} K^{\xi_k}(x_i)+\hat{r}_t(x_i),
    \end{aligned}
\end{equation*}
where the remainder takes the form
\begin{equation}
   \hat{r}_t(x_i):=\frac{1}{2}\mathbb{E}\left[ \left(\bar{X}^i_{T_k}-\bar{X}^i_{t}\right)^{\otimes 2}: \int_0^1 \nabla_{x_i}^2 K^{\xi_k}\left((1-s) \bar{X}_t^i+s \bar{X}_{T_k}^i\right) d s\mid \bar{\textbf{X}}_t = \textbf{x}, \xi_k\right] 
\end{equation}
In Lemma \ref{reminder}, we show that for $t \in\left[T_k, T_{k+1}\right)$,
$$
\mathbb{E}\left[\left|\hat{r}_t\left(\bar{X}_t^i\right)\right|^2 \mid \xi_k\right] \leq c \left(t-T_k\right)^2,
$$
where $c$ is a positive constant independent of $N$ and $\xi_k$. Repeat the same procedure as in step 2, one can get the $\tau^2$ estimate, that is
\begin{equation}
\begin{aligned}
    &\frac{1}{N}\sum_{i=1}^N \mathbb{E}_{\xi_k} \int_{\mathbb{R}^{Nd}} \left| \bar{K}^{\xi_k,i}_t(\textbf{x})-K^{\xi_k}(x_i) \right|^2 \bar{\rho}^{N,\xi_k}_t d \textbf{x} \\  \lesssim &  \mathbb{E}_{\xi_k} \left[\frac{1}{N}\sum_{i=1}^N \mathbb{E}\left[\left|\hat{r}_t\left(\bar{X}_t^i\right)\right|^2 \mid \xi_k \right] +\frac{1}{N}\sum_{i=1}^N \int_{\mathbb{R}^{Nd}} \left|\mathbb{E} \left[\bar{X}^i_{T_k}-\bar{X}^i_{t}\mid\bar{\textbf{X}}_t = \textbf{x}, \xi_k\right]\right|^2 \bar{\rho}^{N,\xi_k}_t d \textbf{x}\right]\\ 
    \leq & \tilde{c}^{\prime} \tau^2\left(3+\frac{1}{N}\mathcal{I}\left(\bar{\rho}_{T_k}^{N}\right)\right),
\end{aligned}
\end{equation}
where the positive constant $\tilde{c}^{\prime}$ is independent of $N$, $\xi_k$ and $t$. \\

\textbf{Step 4:} Estimate $\mathbb{E}_{\xi_k, \tilde{\xi}_k} \left[ \int_{\mathbb{R}^{Nd}} \left| K^{\xi_k}(x_i)- F(x_i)\right|^2 \frac{\left|\bar{\rho}^{N,\xi_k}_t-\bar{\rho}^{N,\tilde{\xi}_k}_t\right|^2}{\bar{\rho}^{N,\tilde{\xi}_k}_t}d\textbf{x}\right]$.\\

By the boundedness assumption in Assumption \ref{assume}, we only need to estimate $\int_{\mathbb{R}^{Nd}} \frac{\left|\bar{\rho}_t^{N, \tilde{\xi}_k}-\bar{\rho}_t^{N, \xi_k}\right|^2}{\bar{\rho}_t^{N, \tilde{\xi}_k}} d \textbf{x}$. 
First note that
$$
\int_{\mathbb{R}^{Nd}} \frac{\left|\bar{\rho}_t^{N, \tilde{\xi}_k}-\bar{\rho}_t^{N, \xi_k}\right|^2}{\bar{\rho}_t^{N, \tilde{\xi}_k}} d \textbf{x}=\int_{\mathbb{R}^{Nd}}\left|\frac{\bar{\rho}_t^{N, \xi_k}}{\bar{\rho}_t^{N, \tilde{\xi}_k}}-1\right|^2 \bar{\rho}_t^{N, \tilde{\xi}_k} d \textbf{x}.
$$
We call any measure on a path space  a path measure. Taking $\phi = \omega_t$, the time projection mapping, then by the Lemma \ref{path} and the Girsanov transform we can conclude
$$
\begin{aligned}
  & \frac{\bar{\rho}_t^{N, \xi_k}}{\bar{\rho}_t^{N, \tilde{\xi}_k}}(\textbf{x})  =  \mathbb{E}\left[\frac{d P_{\bar{\textbf{X}}}}{d P_{\bar{\textbf{Y}}}} \mid \bar{\textbf{Y}}_t=\textbf{x}, \xi_k, \tilde{\xi}_k\right]\\
 = & \mathbb{E}\left[\exp \left(\sqrt{\frac{1}{2\sigma}} \int_{T_k}^{t}\left(\textbf{K}^{\tilde{\xi}_k}-\textbf{K}^{\xi_k}\right)(\textbf{y}) d \textbf{W}_s-\frac{1}{4\sigma} \int_{T_k}^{t}\left|\left(\textbf{K}^{\tilde{\xi}_k}-\textbf{K}^{\xi_k}\right)(\textbf{y})\right|^2 d s\right)\mid \bar{\textbf{Y}}_t=\textbf{x}, \xi_k, \tilde{\xi}_k\right].
\end{aligned}
$$
Denote
$$
\tilde{\textbf{K}}(\textbf{y}):=\frac{1}{\sqrt{2\sigma}}\left(\textbf{K}^{\tilde{\xi}_k}-\textbf{K}^{\xi_k}\right)(\textbf{y}).
$$
Then, for $t \in\left[T_k, T_{k+1}\right)$, 
$$ 
\begin{aligned}
& \int_{\mathbb{R}^{Nd}}\left|\frac{\bar{\rho}_t^{N, \xi_k}}{\bar{\rho}_t^{N, \tilde{\xi_k}}}-1\right|^2 \bar{\rho}_t^{N, \tilde{\xi}_k} d \textbf{x} \\
= & \int_{\mathbb{R}^{Nd}} \bar{\rho}_t^{N, \tilde{\xi}_k}(\textbf{x})\left|\mathbb{E}\left[\exp \left(\int_{T_k}^t \tilde{\textbf{K}}\left(\bar{\textbf{Y}}_{T_k}\right) d \textbf{W}_s-\frac{1}{2} \int_{T_k}^t\left|\tilde{\textbf{K}}\left(\bar{\textbf{Y}}_{T_k}\right)\right|^2 d s\right) \mid \bar{\textbf{Y}}_t=\textbf{x}, \xi_k, \tilde{\xi}_k\right]-1\right|^2 d \textbf{x}.
\end{aligned}
$$
Notice the fact that $\bar{\textbf{Y}}_t=\textbf{y} + \left(\textbf{b}(\textbf{y}) +  \textbf{K}^{\tilde{\xi}_k}(\textbf{y})\right)\left(t-T_k\right)+\sqrt{2 \sigma}\left(\textbf{W}_t-\textbf{W}_{T_k}\right)$, resulting in that
\begin{equation}\label{step3K}
\begin{aligned}
& \int_{\mathbb{R}^{Nd}}\left|\frac{\bar{\rho}_t^{N, \xi_k}}{\bar{\rho}_t^{N, \tilde{\xi_k}}}-1\right|^2 \bar{\rho}_t^{N, \tilde{\xi}_k} d \textbf{x} \\
= & \int \bar{\rho}_t^{N, \tilde{\xi}_k}(\textbf{x}) \left|\int\left(e^{\sqrt{\frac{1}{2\sigma}} \tilde{\textbf{K}}(\textbf{y})\left(\textbf{x}-\textbf{y}-\left(t-T_k\right) \textbf{F}^{\tilde{\xi}_k}(\textbf{y})\right)-\frac{1}{2}|\tilde{\textbf{K}}(\textbf{y})|^2\left(t-T_k\right)}-1\right) P\left(\bar{\textbf{X}}_{T_k}=\textbf{y} \mid \bar{\textbf{Y}}_t=\textbf{x}, \xi_k, \tilde{\xi}_k\right) d \textbf{y}\right|^2 d \textbf{x},
\end{aligned}
\end{equation}
where $\textbf{F}^{\tilde{\xi}_k}(\textbf{y})=\textbf{b}(\textbf{y}) +  \textbf{K}^{\tilde{\xi}_k}(\textbf{y})$. Now we split (\ref{step3K}) into three terms
\begin{equation}
    \begin{aligned}
        & e^{\sqrt{\frac{1}{2\sigma}} \tilde{\textbf{K}}(\textbf{y})\left(\textbf{x}-\textbf{y}-\left(t-T_k\right) \textbf{F}^{\tilde{\xi}_k}(\textbf{y})\right)-\frac{1}{2}|\tilde{\textbf{K}}(\textbf{y})|^2\left(t-T_k\right)}-1\\
        = & \sqrt{\frac{1}{2\sigma}} \tilde{\textbf{K}}(\textbf{y})\left(\textbf{x}-\textbf{y}\right) + \left(-\sqrt{\frac{1}{2\sigma}} \tilde{\textbf{K}}(\textbf{y})\textbf{F}^{\tilde{\xi}_k}(\textbf{y})- \frac{1}{2}|\tilde{\textbf{K}}(\textbf{y})|^2\right)\left(t-T_k\right) + \left(e^{\textbf{z}}-\textbf{z}-1\right)\\
        := & K_1 + K_2 + K_3,
    \end{aligned}
\end{equation}
where $\textbf{z} = \sqrt{\frac{1}{2\sigma}} \tilde{\textbf{K}}(\textbf{y})\left(\textbf{x}-\textbf{y}-\left(t-T_k\right) \textbf{F}^{\tilde{\xi}_k}(\textbf{y})\right)-\frac{1}{2}|\tilde{\textbf{K}}(\textbf{y})|^2\left(t-T_k\right).$ Then
$$ 
 \int_{\mathbb{R}^{Nd}}\left|\frac{\bar{\rho}_t^{N, \xi_k}}{\bar{\rho}_t^{N, \tilde{\xi_k}}}-1\right|^2 \bar{\rho}_t^{N, \tilde{\xi}_k} d \textbf{x} = \int_{\mathbb{R}^{Nd}} \bar{\rho}_t^{N, \tilde{\xi}_k}(\textbf{x})\left|\int_{\mathbb{R}^{Nd}} \left(K_1 + K_2 + K_3\right)P\left(\bar{\textbf{X}}_{T_k}=\textbf{y} \mid \bar{\textbf{Y}}_t=\textbf{x}, \xi_k, \tilde{\xi}_k\right) d \textbf{y}\right|^2 d \textbf{x}.
$$
Equipped with the estimation for integration of $K_1$, $K_2$ and $K_3$ from Lemma \ref{remainK1}, Lemma \ref{remainK2} and  Lemma \ref{remainK3} in Appendix \ref{append}, we can get the estimate
$$
\int_{\mathbb{R}^{Nd}} \frac{\left|\bar{\rho}_t^{N, \tilde{\xi}_k}-\bar{\rho}_t^{N, \xi_k}\right|^2}{\bar{\rho}_t^{N, \tilde{\xi}_k}} d \textbf{x} \leq c\tau^2 \left(N+\mathcal{I}\left(\bar{\rho}^{N}_{T_k}\right)\right).
$$
Hence, 
\begin{equation}
\begin{aligned}
    & \frac{1}{N}\sum_{i=1}^N\mathbb{E}_{\xi_k, \tilde{\xi}_k} \left[ \int_{\mathbb{R}^{Nd}} \left| K^{\xi_k}(x_i)- F(x_i)\right|^2 \frac{\left|\bar{\rho}^{N,\xi_k}_t-\bar{\rho}^{N,\tilde{\xi}_k}_t\right|^2}{\bar{\rho}^{N,\tilde{\xi}_k}_t} d \textbf{x} \right] \\
    \leq & \frac{1}{N} \mathbb{E}_{\xi_k, \tilde{\xi}_k} \left[ \int_{\mathbb{R}^{Nd}} \left| \textbf{K}^{\xi_k}(\textbf{x})- \textbf{F}(\textbf{x})\right|^2 \frac{\left|\bar{\rho}^{N,\xi_k}_t-\bar{\rho}^{N,\tilde{\xi}_k}_t\right|^2}{\bar{\rho}^{N,\tilde{\xi}_k}_t} d \textbf{x} \right] \\
    \lesssim & \tilde{c}^{\prime \prime }\tau^2 \left(1+\frac{1}{N}\mathcal{I}\left(\bar{\rho}^{N}_{T_k}\right)\right),   
\end{aligned}
\end{equation}
where the positive constant  $\tilde{c}^{\prime \prime }$ is independent of $\xi_k$, $N$ and $t$. 

\textbf{Step 5:} Estimate $\int_{\mathbb{R}^{Nd}} \bar{\rho}_t^N \left|K(x_i)-K*\rho_t(x_i)\right|^2 d \textbf{x}$.\\

Let $\Phi(x_i)= \left(F(x_i)-K*\rho_t(x_i)\right)^2$ with $F(x_i)$ given in (\ref{Fxi}). Define the probability density
$$
f = \frac{1}{\lambda} \exp(N \eta \Phi) \rho^N_t, \quad \lambda = \int_{\mathbb{R}^{Nd}} \exp(N \eta \Phi) \rho^N_t d\textbf{x},
$$
where $\eta$ is an arbitrary positive number. Thanks to the function $h(x) = x\log x + 1 - x \geq 0$ for any $x>0$, we can obtain
$$
\begin{aligned}
\frac{1}{N} \int_{\mathbb{R}^{Nd}} \bar{\rho}_t^N \log \bar{\rho}_t^N d\textbf{x} = & \frac{1}{N}\int_{\mathbb{R}^{Nd}} f \left(\frac{\bar{\rho}_t^N}{f} \log \frac{\bar{\rho}_t^N}{f}-\frac{\bar{\rho}_t^N}{f}+1\right) d\textbf{x} + \frac{1}{N} \int_{\mathbb{R}^{Nd}} \bar{\rho}_t^N \log f d\textbf{x} \\
\geq & \frac{1}{N} \int_{\mathbb{R}^{Nd}} \bar{\rho}_t^N \log f d\textbf{x}.
\end{aligned}
$$
On the other hand, one can easily check that
$$
\frac{1}{N} \int_{\mathbb{R}^{Nd}} \bar{\rho}_t^N \log f d\textbf{x} = \eta \int_{\mathbb{R}^{Nd}} \bar{\rho}_t^N \Phi d\textbf{x}+ \frac{1}{N}\int_{\mathbb{R}^{Nd}} \bar{\rho}_t^N \log \rho_t^N d\textbf{x}-\frac{\log \lambda}{N}.
$$
Then we can conclude that
\begin{equation}
    \int_{\mathbb{R}^{Nd}} \bar{\rho}_t^N \Phi d\textbf{x} \leq \frac{1}{\eta}\left(\mathcal{H}_N\left(\bar{\rho}_t^N \mid \rho_t^N\right) + \frac{1}{N} \log \int_{\mathbb{R}^{Nd}} \rho_t^N \exp(N \eta \Phi) d\textbf{x}\right).
\end{equation}
We rewrite 
$$\Phi(x_i)=\left|\frac{1}{N}\sum_{j=1}^N K(x_i-x_j)-K*\rho_t(x_i)\right|^2=\frac{1}{N^2}\sum_{j_1,j_2=1}^N \phi(x_i,x_{j_1}) \phi(x_i,x_{j_2}),$$
where $\phi(x,z)=K(x-z)-K*\rho_t(x).$ Note that each $\phi$ has vanishing expectation with respect to $\rho$:
$$\int \phi(x,z) \rho_t(x) dx =0.$$
By taking $\|K\|_{L^{\infty}}$ small enough, we deduce that 
\begin{equation}
    \int_{\mathbb{R}^{Nd}} \rho^N_t e^{N \eta \Phi} d \textbf{x}=\int_{\mathbb{R}^{Nd}} \rho^N_t \operatorname{exp} \left( \frac{\eta}{N}\sum_{j_1,j_2=1}^N \phi(x_i,x_{j_1}) \phi(x_i,x_{j_2}) \right)d \textbf{x} \leq C,
\end{equation}
where the constant $C$ is independent of $N$ and $t$. Taking $\eta = \frac{4C_{LS}}{\sigma}$, and by Lemma \ref{lln}, we can obtain the estimate:
\begin{equation}
    \int_{\mathbb{R}^{Nd}} \bar{\rho}_t^N \left|F(x_i)-K*\rho_t(x_i)\right|^2 d \textbf{x} \leq \frac{\sigma}{4C_{LS}} \mathcal{H}_{N}\left(\bar{\rho}_t^N \mid \rho_t^N\right) + \frac{C}{N}.
\end{equation}

\textbf{Step 6}: Finally, combining the results in the previous steps we get the following desired estimate for $t\in\left[T_k, T_{k+1}\right)$: 
$$
\begin{aligned}
\frac{d}{d t}   \mathcal{H}_N\left( \bar{\rho}^N_t \mid \rho^N_t \right) \leq & -\frac{\sigma}{4 C_{LS}} \mathcal{H}_N\left( \bar{\rho}^N_t \mid \rho^N_t \right) + c_1 \tau^2 \left(1+ \frac{1}{N}\mathcal{I}\left(\bar{\rho}^{N}_{T_k}\right)\right)+ \frac{c_2}{N} \\
\leq & -\frac{\sigma}{4 C_{LS}} \mathcal{H}_N\left( \bar{\rho}^N_t \mid \rho^N_t \right) + c_1 \tau^2 \left(1+ \frac{1}{N}\mathbb{E}_{\xi_k}\left[\mathcal{I}\left(\bar{\rho}^{N,\xi_k}_{T_k}\right)\right]\right)+ \frac{c_3}{N}\\
\leq & -C_0 \mathcal{H}_N\left( \bar{\rho}^N_t \mid \rho^N_t \right) + C_1 \tau^2 + \frac{C_2}{N},
\end{aligned}
$$
where the second inequality is due to the convexity of the Fisher information $\mathcal{I}(\rho)$ with respect to $\rho$, and the third inequality is due to Lemma \ref{fishlemma}. Note that the constants $C_0$, $C_1$ and $C_2$ are independent of $N$, $\tau$ and $\xi_k$, then by Gronwall's inequality, we end the proof. 
\end{proof}

\section{Conclusion}\label{sec13}

In this paper we prove the error estimate of propagation of chaos form for random batch method for $N-$ particle systems toward its mean-field limit, the Fokker-Planck equation, based on the relative entropy. We show that the convergence rate is $O(\tau^2+1/N)$, where $\tau$ is the small time steps.  Our result can be seen as an improvement over the previous work, and we fill the gap to understand the approximation error of the RBM as a numerical method for its mean-field limit. Our results need some regularity assumptions on the force
terms. It will be an interesting topic to investigate the error estimate of random batch method for singular interacting kernel such as the Biot-Savart Law in incompressible flows, among other applications.

\section*{Acknowledgement}
This work is financially supported by the National Key R\&D Program of China, Project Number 2021YFA1002800. S. Jin was partially supported by the NSFC grant No. 12350710181, The Science and Technology Commission of Shanghai Municipality
grant No. 20JC1414100. The authors are also supported by the Fundamental Research Funds for the Central Universities. The work of L. Li was partially supported by NSFC 12371400,  Shanghai Science and Technology Commission (Grant No. 21JC1403700, 21JC1402900), the Strategic Priority Research Program of Chinese Academy of Sciences, Grant No. XDA25010403. S. Jin and L. Li were also partially supported by Shanghai Municipal Science and Technology Major Project 2021SHZDZX0102.

\appendix
\section{Proof of Lemma \ref{Lioulemma}}\label{liouappend}
\begin{proof}[Proof of Lemma \ref{Lioulemma}]
Indeed, for $t \in\left[T_k, T_{k+1}\right)$, consider the random variable $\bar{\varrho}_t^{N, \boldsymbol{\xi}} \mid \mathcal{F}_{T_k}$, where $\mathcal{F}_{T_k}=\sigma\left(\bar{\textbf{X}}_s, s \leq T_k\right)$. Then, by definition of $\bar{\textbf{X}}_t$ for $t \in\left[T_k, T_{k+1}\right)$, $\bar{\varrho}_t^{N, \boldsymbol{\xi}} \mid \mathcal{F}_{T_k}$ satisfies the Liouville equation:
$$
\partial_t\left(\bar{\varrho}_t^{N, \boldsymbol{\xi}} \mid \mathcal{F}_{T_k}\right)=-\sum_{i=1}^N \operatorname{div}_{x_i} \left(\left(b\left(\bar{X}_{T_k}^i\right)+K^{\xi_k}\left(\bar{X}_{T_k}^i\right)\right) \left(\hat{\rho}_t^{N} \mid \mathcal{F}_{T_k}\right)\right)+
\sum_{i=1}^N \sigma \Delta_{x_i} \left(\bar{\varrho}_t^{N, \boldsymbol{\xi}} \mid \mathcal{F}_{T_k}\right).
$$
Taking expectation, one has
$$
\mathbb{E}\left[\partial_t\left(\bar{\varrho}_t^{N, \boldsymbol{\xi}} \mid \mathcal{F}_{T_k}\right)\right]=\partial_t \bar{\varrho}_t^{N, \boldsymbol{\xi}}, \quad \mathbb{E}\left[\Delta_{x_i} \partial_t\left(\bar{\varrho}_t^{N, \boldsymbol{\xi}} \mid \mathcal{F}_{T_k}\right)\right]=\Delta_{x_i} \bar{\varrho}_t^{N, \boldsymbol{\xi}},
$$
and for the drift term,
$$
\begin{aligned}
&\mathbb{E}\left[\operatorname{div}_{x_i} \left(\left(b\left(\bar{X}_{T_k}^i\right)+K^{\xi_k}\left(\bar{X}_{T_k}^i\right)\right) \left(\bar{\varrho}_t^{N, \boldsymbol{\xi}} \mid \mathcal{F}_{T_k}\right)\right)\right]\\
 = &\operatorname{div}_{x_i} \int\left(\bar{\varrho}_t^{N, \boldsymbol{\xi}} \mid \mathcal{F}_{T_k}\right)(\textbf{x} \mid \textbf{y}) (b(y_i)+K^{\xi_k}(y_i)) \bar{\varrho}_t^{N, \boldsymbol{\xi}}(\textbf{y}) d \textbf{y} \\
 = &\operatorname{div}_{x_i}\int (b(y_i)+K^{\xi_k}(y_i)) \bar{\varrho}_{t,T_k}^{N, \boldsymbol{\xi}}(\textbf{x}, \textbf{y}) d \textbf{y} \\
 = &\operatorname{div}_{x_i} \left(\bar{\varrho}_t^{N, \boldsymbol{\xi}}(\textbf{x}) \mathbb{E}\left[b\left(\bar{X}_{T_k}^i\right) + K^{\xi_k}\left(\bar{X}_{T_k}^i\right) \mid \bar{\textbf{X}}_t=\textbf{x},\boldsymbol{\xi}\right]\right),
\end{aligned}
$$
where $\bar{\varrho}_{t,T_k}^{N, \boldsymbol{\xi}}$ denotes the joint distribution  of $\bar{\textbf{X}}_t$ and $\bar{\textbf{X}}_{T_k}$. Note that we used Bayes' law in the third equality. Combining all the above, we obtain the desired result (\ref{liou}).
\end{proof}
\section{Omitted details in Section \ref{mainthm}}\label{append}
In this Appendix, we prove the details omitted in the proof of Theorem \ref{thm}.

\begin{lemma}\label{reminder}
There exist positive constants $c_1$ and $c_2$ independent of $N$ and the batch $\xi_k$ such that for all $t\in\left[T_k,T_{k+1}\right)$, it holds that
   \begin{equation}
       \mathbb{E} \left[\left|\bar{r}_t\left(\bar{X}_t^i\right) \right|^2 \mid \xi_k\right]\leq c_1(t-T_k)^2;
   \end{equation} 
   \begin{equation}
       \mathbb{E}\left[\left|\hat{r}_t\left(\bar{X}_t^i\right)\right|^2 \left.\right|\xi_k\right] \leq c_2 \left(t-T_k\right)^2.
   \end{equation}
\end{lemma}
\begin{proof}
 Since  the Hessians of $b$ and $K^{\xi_k}$ have polynomial growth and , combining with the moment control yields
 \begin{equation}
 \begin{aligned}
         \left|\bar{r}_t(x_i) \right| & \leq \int_0^1 \mathbb{E}\left[\left|\nabla^2_{x_i} b\left((1-s) \bar{X}_t^i +s \bar{X}_{T_k}^i \right)\right| \cdot\left|\bar{X}_{T_k}^i-\bar{X}_t^i\right|^2 \mid \bar{\textbf{X}}_t=\textbf{x},\xi_k\right]\\
         & \leq \frac{1}{q+1} \mathbb{E} \left[(1+\left|\bar{X}_t^i\right|^q + \left|\bar{X}_{T_k}^i\right|^q)\cdot\left|\bar{X}_{T_k}^i-\bar{X}_t^i\right|^2 \mid \bar{\textbf{X}}_t=\textbf{x},\xi_k\right]\\
         & \leq C \mathbb{E} \left[\left|\bar{X}^i_{T_k}-\bar{X}^i_{t}\right|^2\mid \bar{\textbf{X}}_t=\textbf{x},\xi_k \right]; \\
         \left|\hat{r}_t(x_i) \right| & \leq \int_0^1 \mathbb{E}\left[\left|\nabla^2_{x_i} K^{\xi_k}\left((1-s) \bar{X}_t^i +s \bar{X}_{T_k}^i \right)\right| \cdot\left|\bar{X}_{T_k}^i-\bar{X}_t^i\right|^2 \mid \bar{\textbf{X}}_t=\textbf{x}, \xi_k\right]\\
         &  \leq \tilde{C} \mathbb{E} \left[\left|\bar{X}^i_{T_k}-\bar{X}^i_{t}\right|^2\mid \bar{\textbf{X}}_t=\textbf{x},\xi_k \right].
 \end{aligned}
 \end{equation}
Hence by Jensen's inequality, taking the expectation and using polynomial growth assumption for $b$ and  Lipschitz assumption for $K^{\xi_k}$, we have
\begin{equation}
    \begin{aligned}
\mathbb{E}\left[\left|\bar{r}_t\left(\bar{X}_t^i\right)\right|^2 \mid \xi_k\right] & \leq C^2  \mathbb{E}\left[\left|\mathbb{E}\left[\left|\bar{X}_{T_k}^i-\bar{X}_t^i\right|^2 \mid \bar{\textbf{X}}_t \right]\right|^2 \mid \xi_k\right] \\
& \leq C^2  \mathbb{E}\left[\left|b\left(\bar{X}_{T_k}^i\right)\left(t-T_k\right)+K^{\xi_k}\left(\bar{X}_{T_k}^i\right)\left(t-T_k\right)+\int_{T_k}^t d W\right|^4 \mid \xi_k \right] \\
& \leq C^2 \left(\left(t-T_k\right)^4\left(1+\mathbb{E}\left[\left|\bar{X}_{T_k}^i\right|^q\mid \xi_k\right]\right)^4+\left(t-T_k\right)^4\left(\left|K^{\xi_k}(0)\right|+L \mathbb{E}\left[\left|\bar{X}_{T_k}^i\right|\mid \xi_k\right]\right)^4\right.\\
&  + \left. 3\left(t-T_k\right)^2 \right);\\
\mathbb{E}\left[\left|\hat{r}_t\left(\bar{X}_t^i\right)\right|^2 \mid \xi_k \right] & \leq \tilde{C}^2  \mathbb{E}\left[\left|\mathbb{E}\left[\left|\bar{X}_{T_k}^i-\bar{X}_t^i\right|^2 \mid \bar{\textbf{X}}_t \right]\right|^2 \mid \xi_k \right] \\
& \leq \tilde{C}^2  \mathbb{E}\left[\left|b\left(\bar{X}_{T_k}^i\right)\left(t-T_k\right)+K^{\xi_k}\left(\bar{X}_{T_k}^i\right)\left(t-T_k\right)+\int_{T_k}^t d W\right|^4\mid \xi_k \right] \\
& \leq  \tilde{C}^2 \left(\left(t-T_k\right)^4\left(1+ \mathbb{E}\left[\left|\bar{X}_{T_k}^i\right|^q\mid \xi_k\right]\right)^4+\left(t-T_k\right)^4\left(\left|K^{\xi_k}(0)\right|+L \mathbb{E}\left[\left|\bar{X}_{T_k}^i\right|\mid \xi_k\right]\right)^4\right. \\
&  + \left. 3\left(t-T_k\right)^2 \right).
\end{aligned}
\end{equation}
Finally, by the moment control Lemma \ref{moment}, we have a uniform bound for the moment $\mathbb{E}\left[\left|\bar{X}_{T_k}^i\right|^4\mid \xi_k\right]$, which leads to the conclusion.
\end{proof}

\begin{lemma}\label{remainb}
Under the setting of Theorem \ref{thm}, there exists a positive constant $c$ independent of $\xi_k$, $N$ and $t$ such that for $t\in\left[T_k, T_{k+1}\right)$, 
    \begin{equation}
    \int_{\mathbb{R}^{Nd}} \left|\mathbb{E} \left[\bar{\textbf{X}}_{T_k}-\bar{\textbf{X}}_{t}\mid\bar{\textbf{X}}_t = \textbf{x},\xi_k\right]\right|^2 \bar{\rho}^{N,\xi_k}_t d \textbf{x} \leq c\tau^2\left(3N+\mathcal{I}\left(\bar{\rho}_{T_k}^{N}\right)\right),
\end{equation}
where $\mathcal{I}\left(\bar{\rho}_{T_k}^{N}\right)$ is the Fisher information.
\end{lemma}
\begin{proof}
    We split (\ref{step2bayes}) into three parts and use integration by parts. Let
$$
\begin{aligned}
\textbf{y}-\textbf{x}= & \left(\textbf{I}_{Nd}+\left(t-T_k\right) \left(\nabla \textbf{b}(\textbf{y})+\nabla \textbf{K}^{\xi_k}(\textbf{y})\right)\right) \cdot\left(\textbf{y}-\textbf{x}+\left(t-T_k\right)\left(\textbf{b}(\textbf{y})+ \textbf{K}^{\xi_k}(\textbf{y})\right)\right) \\
& -\left(t-T_k\right) \left(\nabla \textbf{b}(\textbf{y})+\nabla \textbf{K}^{\xi_k}(\textbf{y})\right) \cdot\left(\textbf{y}-\textbf{x}+\left(t-T_k\right)\left(\textbf{b}(\textbf{y})+ \textbf{K}^{\xi_k}(\textbf{y})\right)\right) \\
& -\left(t-T_k\right)\left(\textbf{b}(\textbf{y})+ \textbf{K}^{\xi_k}(\textbf{y})\right) \\
:= & a_1(\textbf{x},\textbf{y})-a_2(\textbf{x},\textbf{y})-a_3(\textbf{x},\textbf{y}),
\end{aligned}
$$
and define
$$
I_i(\textbf{x}):=\mathbb{E}\left[a_i\left(\bar{\textbf{X}}_t, \bar{\textbf{X}}_{T_k}\right) \mid \bar{\textbf{X}}_t=\textbf{x},\xi_k\right], \quad i=1,2,3.
$$
\begin{itemize}
    \item [(a)] For the term $I_1$, since the distribution $P\left(\bar{\textbf{X}}_t=\textbf{x} \mid \bar{\textbf{X}}_{T_k}=\textbf{y},\xi_k\right)$ is Gaussian, then after integration by parts we obtain:
$$
I_1(\textbf{x})=2 \sigma\left(t-T_k\right) \int_{\mathbb{R}^{Nd}} \frac{\nabla \bar{\rho}_{T_k}^{N}(\textbf{y})}{\bar{\rho}_t^{N,\xi_k}(\textbf{x})} P\left(\bar{\textbf{X}}_t=\textbf{x} \mid \bar{\textbf{X}}_{T_k}=\textbf{y} ,\xi_k \right) d \textbf{y}.
$$
Using Bayes' law again, one has
$$
  \begin{aligned}
      I_1(\textbf{x})= & 2 \sigma\left(t-T_k\right) \int_{\mathbb{R}^{Nd}} \frac{\nabla \bar{\rho}_{T_k}^{N}(\textbf{y})}{\bar{\rho}_{T_k}^{N}(\textbf{y})} P\left( \bar{\textbf{X}}_{T_k}=\textbf{y} \mid \bar{\textbf{X}}_t=\textbf{x},\xi_k \right) d \textbf{y}.
  \end{aligned}  
$$
Hence, by Jensen's inequality,
$$
\begin{aligned}
\mathbb{E}\left[\left|I_1\left(\bar{\textbf{X}}_t\right)\right|^2 \mid \xi_k \right] & \leq 4 \sigma^{2}\left(t-T_k\right)^2 \int_{\mathbb{R}^{Nd}} \bar{\rho}_t^{N,\xi_k}(\textbf{x}) \int_{\mathbb{R}^{Nd}}  \left|\frac{\nabla \bar{\rho}_{T_k}^{N}(\textbf{y})}{\bar{\rho}_{T_k}^{N}(\textbf{y})}\right|^2 P\left( \bar{\textbf{X}}_{T_k}=\textbf{y} \mid \bar{\textbf{X}}_t=\textbf{x},\xi_k \right) d \textbf{y} d \textbf{x} \\
& :=c\left(t-T_k\right)^2 \mathcal{I}\left(\bar{\rho}_{T_k}^N\right) .
\end{aligned}
$$

\item[(b)] For the term $I_2$, using the Lipshitz condition in Assumption \ref{assume} and Jensen's inequality,
$$
\begin{aligned}
\mathbb{E}\left[\left|I_2\left(\bar{\textbf{X}}_t\right)\right|^2 \mid \xi_k\right] & \leq c^{\prime}\left(t-T_k\right)^2 \mathbb{E}\left[\left|\mathbb{E}\left[\bar{\textbf{X}}_{T_k}-\bar{\textbf{X}}_t+\left(t-T_k\right)\left(\textbf{b}\left(\bar{\textbf{X}}_{T_k}\right)+ \textbf{K}^{\xi_k}\left(\bar{\textbf{X}}_{T_k}\right)\right) \mid \bar{\textbf{X}}_{t}\right]\right|^2 \mid \xi_k\right] \\
& \leq c^{\prime}\left(t-T_k\right)^2 \mathbb{E}\left[\left|\bar{\textbf{X}}_{T_k}-\bar{\textbf{X}}_t+\left(t-T_k\right)\left(\textbf{b}\left(\bar{\textbf{X}}_{T_k}\right)+ \textbf{K}^{\xi_k}\left(\bar{\textbf{X}}_{T_k}\right)\right)\right|^2 \mid \xi_k\right] \\
& \leq c^{\prime}\left(t-T_k\right)^2 \mathbb{E}\left[\left|\int_{T_t}^t \sqrt{2 \sigma} d \textbf{W}_s\right|^2 \mid \xi_k \right] :=c^{\prime}N \left(t-T_k\right)^3 .
\end{aligned}
$$ 

\item[(c)] For the term $I_3$,  using Jensen's inequality and the polynomial growth assumption for $b$ and boundedness assumption for $K^{\xi_k}$, combining with the moment control Lemma \ref{moment}, it is clear that
$$
\begin{aligned}
\mathbb{E}\left[\left|I_3\left(\bar{\textbf{X}}_t\right)\right|^2 \mid \xi_k\right] & \leq\left(t-T_k\right)^2 \mathbb{E}\left[\left|\textbf{b}\left(\bar{\textbf{X}}_{T_k}\right)+\textbf{K}^{\xi_k}\left(\bar{\textbf{X}}_{T_k}\right)\right|^2 \mid \xi_k\right] \\
& \leq  c^{\prime \prime} N\left(t-T_k\right)^2.
\end{aligned}
$$
\end{itemize}
In conclusion, we obtain an $\tau^2$ estimate for $\mathbb{E}\left[\left|I_i\left(\bar{\textbf{X}}_t\right)\right|^2\right], i=1,2,3$. We claim that, there exist positive constants $c, c^{\prime}, c^{\prime \prime}$ independent of $N$ and $t$ such that for $t \in\left[T_k, T_{k+1}\right)$,
$$
\begin{gathered}
\mathbb{E}\left[\left|I_1\left(\bar{\textbf{X}}_t\right)\right|^2  \mid \xi_k\right] \leq c\left(t-T_k\right)^2 \mathcal{I}\left(\bar{\rho}_{T_k}^{N}\right), \\
\mathbb{E}\left[\left|I_2\left(\bar{\textbf{X}}_t\right)\right|^2  \mid \xi_k\right] \leq c^{\prime}N \left(t-T_k\right)^3, \\
\mathbb{E}\left[\left|I_3\left(\bar{\textbf{X}}_t\right)\right|^2  \mid \xi_k\right] \leq c^{\prime \prime} N\left(t-T_k\right)^2,
\end{gathered}
$$
where $\mathcal{I}\left(\bar{\rho}_{T_k}^{N}\right):=\int_{\mathbb{R}^{Nd}} \bar{\rho}_{T_k}^{N} |\nabla \log \bar{\rho}_{T_k}^{N}|^2 d \textbf{x}$ is the Fisher information of the joint law for the $N$ particles. Hence, combining the above bound one gets
\begin{equation}
 \int_{\mathbb{R}^{Nd}} \left|\mathbb{E} \left[\bar{\textbf{X}}_{T_k}-\bar{\textbf{X}}_{t}\mid\bar{\textbf{X}}_t = \textbf{x},\xi_k\right]\right|^2 \bar{\rho}^{N,\xi_k}_t d \textbf{x} \leq 3\sum_{k=1}^3 \mathbb{E}\left[\left|I_k\left(\bar{\textbf{X}}_t\right)\right|^2  \mid \xi_k\right]  \leq c \tau^2
   \left(3N+\mathcal{I}\left(\bar{\rho}_{T_k}^{N}\right)\right),
\end{equation}
where the positive constant $\tilde{c}$ is independent of $\xi_k$, $N$ and $t$, and this is what we desire.
\end{proof}

\begin{lemma}\label{remainK1}
Recall the notations
\begin{equation}
  K_1 = \sqrt{\frac{1}{2\sigma}} \tilde{\textbf{K}}(\textbf{y})\left(\textbf{x}-\textbf{y}\right), \tilde{\textbf{K}}(\textbf{y}):=\frac{1}{\sqrt{2\sigma}}\left(\textbf{K}^{\tilde{\xi}_k}-\textbf{K}^{\xi_k}\right)(\textbf{y}).
\end{equation}
Then under the setting of Theorem \ref{thm}, there exists a positive constant $c$ independent of $\xi_k$, $\tilde{\xi}_k$, $N$ and $t$ such that for $t\in\left[T_k, T_{k+1}\right)$, 
    \begin{equation}\label{k1}
    \int_{\mathbb{R}^{Nd}} \bar{\rho}_t^{N, \tilde{\xi}_k}(\textbf{x})\left|\int_{\mathbb{R}^{Nd}} K_1 P\left(\bar{\textbf{X}}_{T_k}=\textbf{y} \mid \bar{\textbf{Y}}_t=\textbf{x}, \xi_k, \tilde{\xi}_k\right) d \textbf{y}\right|^2 d \textbf{x} \leq c \tau^2\left(N+\mathcal{I}\left(\bar{\rho}^{N}_{T_k}\right)\right).
\end{equation}
\end{lemma}
\begin{proof}
We spilt the term $K_1$ into three parts:
    $$
    \begin{aligned}
        \tilde{\textbf{K}}(\textbf{y})\left(\textbf{x}-\textbf{y}\right) = & \tilde{\textbf{K}}(\textbf{y}) \cdot \left( \textbf{I}_{Nd} + \left(t-T_k\right)\nabla \textbf{F}^{\tilde{\xi}_k}(\textbf{y}) \right) \cdot\left( \textbf{y}-\textbf{x}+\left(t-T_k\right)\textbf{F}^{\tilde{\xi}_k}(\textbf{y})\right)\\
        & -  \left(t-T_k\right) \tilde{\textbf{K}}(\textbf{y}) \cdot\nabla\textbf{F}^{\tilde{\xi}_k}(\textbf{y})\cdot\left( \textbf{y}-\textbf{x}+\left(t-T_k\right)\textbf{F}^{\tilde{\xi}_k}(\textbf{y})\right)\\
        & -  \left(t-T_k\right) \tilde{\textbf{K}}(\textbf{y}) \cdot \textbf{F}^{\tilde{\xi}_k}(\textbf{y})\\
        := & \bar{a}_1(\textbf{x},\textbf{y})-\bar{a}_2(\textbf{x},\textbf{y})-\bar{a}_3(\textbf{x},\textbf{y}),
    \end{aligned}
    $$
    where $\textbf{F}^{\tilde{\xi}_k}(\textbf{y})=\textbf{b}(\textbf{y}) +  \textbf{K}^{\tilde{\xi}_k}(\textbf{y})$. 
    Define 
    $$
    \bar{I}_j(\textbf{x}) := \mathbb{E} \left[\bar{a}_j(\bar{\textbf{X}}_{T_k},\bar{\textbf{Y}}_t)\mid\bar{\textbf{Y}}_t = \textbf{x}, \xi_k, \tilde{\xi}_k\right], \quad j = 1, 2, 3.
    $$
    Then
    \begin{equation}
         \int_{\mathbb{R}^{Nd}} \bar{\rho}_t^{N, \tilde{\xi}_k}(\textbf{x})\left| \int_{\mathbb{R}^{Nd}} K_1 P\left(\bar{\textbf{X}}_{T_k}=\textbf{y} \mid \bar{\textbf{Y}}_t=\textbf{x}, \xi_k, \tilde{\xi}_k\right) d \textbf{y}\right|^2 d \textbf{x} = \frac{1}{2\sigma} \mathbb{E}\left[\left|\bar{I}_1(\bar{\textbf{Y}}_t)-\bar{I}_2(\bar{\textbf{Y}}_t)-\bar{I}_3(\bar{\textbf{Y}}_t)\right|^2 \mid \xi_k, \tilde{\xi}_k\right].
    \end{equation}
For the first term $ \bar{I}_1$, using Bayes' formula and integration by parts, since $P\left( \bar{\textbf{Y}}_t = \textbf{x}\mid \bar{\textbf{X}}_{T_k}=\textbf{y}, \xi_k, \tilde{\xi}_k\right)$ is Gaussian, we have
$$
\begin{aligned}
 & \bar{I}_1(\textbf{x}) \\
= &  \int_{\mathbb{R}^{Nd}} \tilde{\textbf{K}}(\textbf{y}) \cdot \left( \textbf{I}_{Nd} + \left(t-T_k\right)\nabla \textbf{F}^{\tilde{\xi}_k}(\textbf{y}) \right) \cdot\left( \textbf{y}-\textbf{x}+\left(t-T_k\right)\textbf{F}^{\tilde{\xi}_k}(\textbf{y})\right) P\left( \bar{\textbf{X}}_{T_k}=\textbf{y} \mid \bar{\textbf{Y}}_t = \textbf{x}, \xi_k, \tilde{\xi}_k\right) d \textbf{y}\\
  = & \int_{\mathbb{R}^{Nd}} \tilde{\textbf{K}}(\textbf{y}) \cdot \left( \textbf{I}_{Nd} + \left(t-T_k\right)\nabla \textbf{F}^{\tilde{\xi}_k}(\textbf{y}) \right) \cdot\left( \textbf{y}-\textbf{x}+\left(t-T_k\right)\textbf{F}^{\tilde{\xi}_k}(\textbf{y})\right) \frac{\bar{\rho}^{N, \xi_k}_{T_k}(\textbf{y})}{\bar{\rho}^{N, \tilde{\xi}_k}_{t}(\textbf{x})}P\left( \bar{\textbf{Y}}_t = \textbf{x} \mid \bar{\textbf{X}}_{T_k}=\textbf{y}\right) d \textbf{y}\\
  = & 2\sigma\left(t-T_k\right) \int_{\mathbb{R}^{Nd}} \frac{\nabla_{\textbf{y}}\left( \tilde{\textbf{K}}(\textbf{y})\bar{\rho}^{N, \xi_k}_{T_k}(\textbf{y})\right)}{\bar{\rho}^{N, \tilde{\xi}_k}_{t}(\textbf{x})}P\left( \bar{\textbf{Y}}_t = \textbf{x} \mid \bar{\textbf{X}}_{T_k}=\textbf{y}\right) d \textbf{y}\\
  = & 2\sigma\left(t-T_k\right) \int_{\mathbb{R}^{Nd}} \left( \nabla \tilde{\textbf{K}}(\textbf{y}) + \tilde{\textbf{K}}(\textbf{y}) \frac{\nabla\bar{\rho}^{N, \xi_k}_{T_k}(\textbf{y})}{\bar{\rho}^{N, \xi_k}_{T_k}(\textbf{y})}\right) P\left( \bar{\textbf{Y}}_t = \textbf{x} \mid \bar{\textbf{X}}_{T_k}=\textbf{y}\right) d \textbf{y}.
  \end{aligned}
$$
Since $K^{\xi_k}$ is uniformly bounded by Assumption \ref{assume}, then $\tilde{\textbf{K}}(\textbf{y})$ and $\nabla \tilde{\textbf{K}}(\textbf{y})$ are also uniformly bounded, hence
\begin{equation}\label{bara1}
    \mathbb{E}\left[\left|\bar{I}_1(\bar{\textbf{Y}}_t)\right|^2 \mid \xi_k, \tilde{\xi}_k \right] \leq c \left(t-T_k\right)^2 \left(1+ \mathcal{I}\left(\bar{\rho}^{N}_{T_k}\right)\right),
\end{equation}
and the constant $c$ is independent of $N$, $\xi_k$, $\tilde{\xi}_k$.\\
For the second term $ \bar{I}_2$, by Jensen's inequality, it holds that
$$
\mathbb{E}\left[\left|\bar{I}_2(\bar{\textbf{Y}}_t)\right|^2 \mid \xi_k, \tilde{\xi}_k \right] \leq \left(t-T_k\right)^2 \mathbb{E}\left[\left|\tilde{\textbf{K}}(\bar{\textbf{X}}_{T_k})\cdot\nabla \textbf{F}^{\tilde{\xi}_k}(\bar{\textbf{X}}_{T_k})\cdot \int_{T_k}^{t} d \textbf{W}_s\right|^2 \mid \xi_k, \tilde{\xi}_k\right].
$$ 
By Assumption \ref{assume}, since $\tilde{\textbf{K}}$ is bounded and $\nabla \textbf{F}^{\tilde{\xi}_k}$ has polynomial growth, then by the moment control Lemma \ref{moment}, we can obtain
\begin{equation}\label{bara2}
    \mathbb{E}\left[\left|\bar{I}_2(\bar{\textbf{Y}}_t)\right|^2 \mid \xi_k, \tilde{\xi}_k \right] \leq c N\left(t-T_k\right)^2,
\end{equation}
and the constant $c$ is independent of $N$, $\xi_k$, $\tilde{\xi}_k$.\\
For the third term $ \bar{I}_3$, by Jensen's inequality, it holds that
$$
\mathbb{E}\left[\left|\bar{I}_3(\bar{\textbf{Y}}_t)\right|^2 \mid \xi_k, \tilde{\xi}_k \right] \leq \left(t-T_k\right)^2 \mathbb{E}\left[\left|\tilde{\textbf{K}}(\bar{\textbf{X}}_{T_k})\cdot \textbf{F}^{\tilde{\xi}_k}(\bar{\textbf{X}}_{T_k})\right|^2 \mid \xi_k, \tilde{\xi}_k\right].
$$ 
By Assumption \ref{assume}, since $\tilde{\textbf{K}}$ is bounded and $\textbf{F}^{\tilde{\xi}_k}$ has polynomial growth, then by moment control Lemma \ref{moment}, we can obtain
\begin{equation}\label{bara3}
    \mathbb{E}\left[\left|\bar{I}_3(\bar{\textbf{Y}}_t)\right|^2 \mid \xi_k, \tilde{\xi}_k \right] \leq c N\left(t-T_k\right)^2,
\end{equation}
and the constant $c$ is independent of $N$, $\xi_k$, $\tilde{\xi}_k$. Finally, combining (\ref{bara1}), (\ref{bara2}) and (\ref{bara3}), we get 
\begin{equation}\label{k1}
    \int_{\mathbb{R}^{Nd}} \bar{\rho}_t^{N, \tilde{\xi}_k}(\textbf{x})\left|\int_{\mathbb{R}^{Nd}} K_1 P\left(\bar{\textbf{X}}_{T_k}=\textbf{y} \mid \bar{\textbf{Y}}_t=\textbf{x}, \xi_k, \tilde{\xi}_k\right) d \textbf{y}\right|^2 d \textbf{x} \leq c \tau^2\left(N+\mathcal{I}\left(\bar{\rho}^{N}_{T_k}\right)\right).
\end{equation}
\end{proof}

\begin{lemma}\label{remainK2}
Recall the notations
\begin{equation}
  K_2 = \left(-\sqrt{\frac{1}{2\sigma}} \tilde{\textbf{K}}(\textbf{y})\textbf{F}^{\tilde{\xi}_k}(\textbf{y})- \frac{1}{2}|\tilde{\textbf{K}}(\textbf{y})|^2\right)\left(t-T_k\right).
\end{equation}
Then under the setting of Theorem \ref{thm}, there exists a positive constant $c$ independent of $\xi_k$, $\tilde{\xi}_k$, $N$ and $t$ such for $t\in\left[T_k, T_{k+1}\right)$, 
    \begin{equation}\label{k2}
    \int_{\mathbb{R}^{Nd}} \bar{\rho}_t^{N, \tilde{\xi}_k}(\textbf{x})\left|\int_{\mathbb{R}^{Nd}} K_2 P\left(\bar{\textbf{X}}_{T_k}=\textbf{y} \mid \bar{\textbf{Y}}_t=\textbf{x}, \xi_k, \tilde{\xi}_k\right) d \textbf{y}\right|^2 d \textbf{x} \leq c N\tau^2.
\end{equation}
\end{lemma}
\begin{proof}
    For $K_2$, by Jensen's inequality, 
$$
\begin{aligned}
    & \int \bar{\rho}_t^{N, \tilde{\xi}_k}(\textbf{x})\left|\int K_2 P\left(\bar{\textbf{X}}_{T_k}=\textbf{y} \mid \bar{\textbf{Y}}_t=\textbf{x}, \xi_k, \tilde{\xi}_k\right) d \textbf{y}\right|^2 d \textbf{x} \\
     \leq & \int \bar{\rho}_t^{N, \tilde{\xi}_k}(\textbf{x}) \int \left|K_2\right|^2 P\left(\bar{\textbf{X}}_{T_k}=\textbf{y} \mid \bar{\textbf{Y}}_t=\textbf{x}, \xi_k, \tilde{\xi}_k\right) d \textbf{y} d \textbf{x}\\
     = & \left(t-T_k\right)^2 \mathbb{E}\left[\left|\left(\sqrt{\frac{1}{2\sigma}} \tilde{\textbf{K}}(\textbf{y})\textbf{F}^{\tilde{\xi}_k}(\textbf{y})+ \frac{1}{2}|\tilde{\textbf{K}}(\textbf{y})|^2\right)\right|^2 \mid \xi_k, \tilde{\xi}_k\right].
\end{aligned}
$$
By Assumption \ref{assume}, since $\tilde{\textbf{K}}$ is bounded and $\textbf{F}^{\tilde{\xi}_k}$ has polynomial growth, then by moment control Lemma \ref{moment}, we can obtain
\begin{equation}\label{k2}
    \int \bar{\rho}_t^{N, \tilde{\xi}_k}(\textbf{x})\left|\int K_2 P\left(\bar{\textbf{X}}_{T_k}=\textbf{y} \mid \bar{\textbf{Y}}_t=\textbf{x}, \xi_k, \tilde{\xi}_k\right) d \textbf{y}\right|^2 d \textbf{x} \leq cN \tau^2,
\end{equation}
and the constant $c$ is independent of $N$, $\xi_k$, $\tilde{\xi}_k$.
\end{proof}

\begin{lemma}\label{remainK3}
    Recall the notations
\begin{equation}
  K_3 = \left(e^{\textbf{z}}-\textbf{z}-1\right),
\end{equation}
where $$\textbf{z} = \sqrt{\frac{1}{2\sigma}} \tilde{\textbf{K}}(\textbf{y})\left(\textbf{x}-\textbf{y}-\left(t-T_k\right) \textbf{F}^{\tilde{\xi}_k}(\textbf{y})\right)-\frac{1}{2}|\tilde{\textbf{K}}(\textbf{y})|^2\left(t-T_k\right).$$
Then under the setting of Theorem \ref{thm}, there exists a positive constant $c$ independent of $\xi_k$, $\tilde{\xi}_k$, $N$ and $t$ such for $t\in\left[T_k, T_{k+1}\right)$, 
    \begin{equation}\label{k3}
    \int_{\mathbb{R}^{Nd}} \bar{\rho}_t^{N, \tilde{\xi}_k}(\textbf{x})\left|\int_{\mathbb{R}^{Nd}} K_3 P\left(\bar{\textbf{X}}_{T_k}=\textbf{y} \mid \bar{\textbf{Y}}_t=\textbf{x}, \xi_k, \tilde{\xi}_k\right) d \textbf{y}\right|^2 d \textbf{x} \leq c N\tau^2.
\end{equation}
\end{lemma}
\begin{proof}
For $K_3$, by Jensen's inequality, 
$$
\begin{aligned}
    & \int_{\mathbb{R}^{Nd}} \bar{\rho}_t^{N, \tilde{\xi}_k}(\textbf{x})\left|\int_{\mathbb{R}^{Nd}} K_3 P\left(\bar{\textbf{X}}_{T_k}=\textbf{y} \mid \bar{\textbf{Y}}_t=\textbf{x}, \xi_k, \tilde{\xi}_k\right) d \textbf{y}\right|^2 d \textbf{x} \\
     \leq & \int_{\mathbb{R}^{Nd}} \bar{\rho}_t^{N, \tilde{\xi}_k}(\textbf{x}) \int_{\mathbb{R}^{Nd}} \left|K_3\right|^2 P\left(\bar{\textbf{X}}_{T_k}=\textbf{y} \mid \bar{\textbf{Y}}_t=\textbf{x}, \xi_k, \tilde{\xi}_k\right) d \textbf{y} d \textbf{x}\\
     = & \mathbb{E}\left[\left|e^{\hat{\textbf{Y}}_t}-1 - \hat{\textbf{Y}}_t\right|^2 \mid \xi_k, \tilde{\xi}_k\right],
\end{aligned}
$$
where we denote the process
$$
\begin{aligned}
    \hat{\textbf{Y}}_t := & \sqrt{\frac{1}{2\sigma}} \tilde{\textbf{K}}(\bar{\textbf{X}}_{T_k})\cdot \left(\bar{\textbf{Y}}_{t}-\bar{\textbf{X}}_{T_k}-\left(t-T_k\right) \textbf{F}^{\tilde{\xi}_k}(\bar{\textbf{X}}_{T_k})\right)-\frac{1}{2}|\tilde{\textbf{K}}(\bar{\textbf{X}}_{T_k})|^2\left(t-T_k\right)\\
    = & -\frac{1}{2}|\tilde{\textbf{K}}(\bar{\textbf{X}}_{T_k})|^2\left(t-T_k\right) +  \tilde{\textbf{K}}(\bar{\textbf{X}}_{T_k}) \cdot \int_{T_k}^t d \textbf{W}_s.
    \end{aligned}
$$
Denote $\bar{\textbf{Z}}_t = e^{\hat{\textbf{Y}}_t}-1 - \hat{\textbf{Y}}_t$, then by It\^o's formula,
$$
\bar{\textbf{Z}}_t = \frac{1}{2} \left|\tilde{\textbf{K}}\right|^2 \left(t-T_k\right) + \tilde{\textbf{K}} \cdot \int_{T_k}^t \left(e^{\hat{\textbf{Y}}_s}-1  \right) d\textbf{W}_s = \frac{1}{2} \left|\tilde{\textbf{K}}\right|^2 \left(t-T_k\right) + \tilde{\textbf{K}} \cdot \int_{T_k}^t \left(\bar{\textbf{Z}}_t+\hat{\textbf{Y}}_s  \right) d\textbf{W}_s.
$$
So
$$
\mathbb{E}\left[\left|\bar{\textbf{Z}}_t\right|^2 \mid \xi_k, \tilde{\xi}_k\right] = \frac{1}{4} \mathbb{E}\left[ \left|\tilde{\textbf{K}}\right|^2 \mid \xi_k, \tilde{\xi}_k\right] + \int_{T_k}^t \mathbb{E}\left[\left|\tilde{\textbf{K}}\right|^2  \left(\bar{\textbf{Z}}_t+\hat{\textbf{Y}}_s \right)^2\mid \xi_k, \tilde{\xi}_k\right] ds.
$$
By Assumption \ref{assume}, since $\tilde{\textbf{K}}$ is uniformly bounded, then $\mathbb{E}\left[\left|\hat{\textbf{Y}}_t\right|^2 \mid \xi_k, \tilde{\xi}_k\right] \leq c N\left(t-T_k\right)^2.$
Thus
$$
\mathbb{E}\left[\left|\bar{\textbf{Z}}_t\right|^2 \mid \xi_k, \tilde{\xi}_k\right] \leq c \int_{T_k}^t \mathbb{E}\left[\left|\bar{\textbf{Z}}_s\right|^2 \mid \xi_k, \tilde{\xi}_k\right]  ds + cN\left(t-T_k\right)^2. 
$$
By Gronwall's inequality, 
\begin{equation}\label{k3}
\int_{\mathbb{R}^{Nd}} \bar{\rho}_t^{N, \tilde{\xi}_k}(\textbf{x})\left|\int_{\mathbb{R}^{Nd}} K_3 P\left(\bar{\textbf{X}}_{T_k}=\textbf{y} \mid \bar{\textbf{Y}}_t=\textbf{x}, \xi_k, \tilde{\xi}_k\right) d \textbf{y}\right|^2 d \textbf{x} \leq \mathbb{E}\left[\left|\bar{\textbf{Z}}_t\right|^2 \mid \xi_k, \tilde{\xi}_k\right] \leq cN\tau^2,
\end{equation}
and the constant $c$ is independent of $N$, $\xi_k$, $\tilde{\xi}_k$.
\end{proof}

\section{The Radon-Nikodym derivatives of path measures}

The following lemma describes the relationship between the two Radon-Nikodym derivatives (of path measures and of push forward measures):
\begin{lemma}[Lemma A.1, \cite{li2022sharp}]\label{path}
 Let $Q_1, Q_2$ be two probability distributions on $\mathcal{X}$, and $Q_2$ is absolutely continuous with respect to $Q_1$. Let $\phi: \mathcal{X} \rightarrow \mathbb{R}^d$ be a measurable mapping, and consider the push forward measure $\phi_{\#} Q_1$ and $\phi_{\#} Q_2$, denoted by $Q_{1, \phi}$ and $Q_{2, \phi}$, respectively. Then the Randon-Nikodym derivatives $\frac{d Q_{1, \phi}}{d Q_{2, \phi}} \in L^1\left(d Q_{2, \phi}, \mathbb{R}^d\right), \frac{d Q_1}{d Q_2} \in L^1\left(d Q_2, \mathcal{X}\right)$ are well-defined, and
$$
\frac{d Q_{1, \phi}}{d Q_{2, \phi}}(x)=\mathbb{E}_{X \sim Q_2}\left[\frac{d Q_1}{d Q_2} \mid \phi(X)=x\right], 
$$
\end{lemma}

\section{The Law of Large Numbers at exponential scale}

The following lemma describes the Law of Large Numbers at exponential scale which is appeared in \cite{jabin2018quantitative}:
\begin{lemma}\label{lln}
 Consider any $\rho \in L^1\left(\mathbb{R}^d\right)$ with $\rho \geq 0$ and $\int_{\mathbb{R}^d} \rho(x) \mathrm{d} x=1$. Assume that a scalar function $\psi \in L^{\infty}$ with $\|\psi\|_{L^{\infty}}<\frac{1}{2 e}$, and that for any fixed $z, \int_{\mathbb{R}^d} \psi(z, x) \rho(x) d x=0$, then
$$
\int_{\mathbb{R}^{Nd}} \rho^N \exp \left(\frac{\eta}{N} \sum_{j_1, j_2=1}^N \psi\left(x_1, x_{j_1}\right) \psi\left(x_1, x_{j_2}\right)\right) \mathrm{d} \textbf{x} \leq C,
$$
where $\rho^N\left(t, \textbf{x}\right)=\Pi_{i=1}^N \rho\left(t, x_i\right)$, $\eta$ is an arbitrary positive number and constant $C$ only depends on $\eta$ and $\|\psi\|_{L^{\infty}}$.
\end{lemma}
\bibliographystyle{plain}
\bibliography{sn-bibliography}

\end{document}